\documentclass[12pt,reqno]{amsart}

\usepackage[all]{xypic}
\usepackage{enumerate}
\usepackage{hyperref}
\usepackage{a4}

\DeclareMathOperator{\res}{res}

\DeclareMathOperator{\Area}{Area}

\DeclareMathOperator{\sdet}{sdet}

\DeclareMathOperator{\rk}{rk}

\DeclareMathOperator{\tr}{tr}
\DeclareMathOperator{\str}{str}

\DeclareMathOperator{\Ad}{Ad}

\newcommand\dets{{\det}_*}

\newcommand\RS{\textrm{RS}}
\newcommand\grr{\textrm{gr}}
\newcommand\Hurw{\mathrm{Hurw}}
\newcommand\Riem{\mathrm{Riem}}
\newcommand\Epst{\mathrm{Epst}}

\newcommand\goe{\mathfrak g}

\newcommand\moe{\mathfrak m}
\newcommand\noe{\mathfrak n}
\newcommand\zoe{\mathfrak z}
\newcommand\ii{\mathbf i}

\DeclareMathOperator{\GL}{GL}
\DeclareMathOperator{\Aut}{Aut}

\newcommand\Z{\mathbb Z}
\newcommand\N{\mathbb N}
\newcommand\R{\mathbb R}
\newcommand\C{\mathbb C}
\newcommand\Q{\mathbb Q}
\newcommand\I{\mathrm{I}}
\newcommand\II{\mathrm{I\!I}}
\newcommand\III{\mathrm{I\!I\!I}}

\theoremstyle{plain}
  \newtheorem{theorem}{Theorem}
  \newtheorem*{theorem*}{Theorem}

  \newtheorem*{corollary*}{Corollary}
  \newtheorem{lemma}{Lemma}

\theoremstyle{definition}
  
\theoremstyle{remark}
  \newtheorem{remark}[theorem]{Remark}

\begin{document}

\title[Analytic torsion of nilmanifolds with (2,3,5) distributions]{Analytic torsion of nilmanifolds\\ with (2,3,5) distributions}

\author{Stefan Haller}

\address{Stefan Haller,
         Department of Mathematics,
         University of Vienna,
         Oskar-Morgenstern-Platz 1,
         1090 Vienna,
         Austria.}

\email{stefan.haller@univie.ac.at}

\begin{abstract}
	We consider generic rank two distributions on 5-dimensional nilmanifolds, and show that the analytic torsion of their Rumin complex coincides with the Ray--Singer torsion.
\end{abstract}




\maketitle

\section{Introduction}\label{S:intro}

	The classical Ray--Singer analytic torsion \cite{RS71} is a spectral invariant extracted from the de~Rham complex of a closed manifold.
	The celebrated Cheeger--M\"uller theorem \cite{BZ92,Ch77,Ch79,M78} asserts that this analytic torsion essentially coincides with the Reidemeister torsion, a topological invariant.

	Rumin and Seshadri \cite{RS12} have introduced an analytic torsion of the Rumin complex on contact manifolds \cite{R90,R94,R00} and showed that it coincides with the Ray--Singer torsion for 3-dimensional CR Seifert manifolds.
	Further computations for contact spheres and lens spaces have been carried out by Kitaoka \cite{K20,K22}.
	Recently, Albin and Quan \cite{AQ22} proved that the Rumin--Seshadri analytic torsion differs from the Ray--Singer torsion by the integral of a local quantity, which yet has to be identified explicitly.

	An analytic torsion for the Rumin complex on more general filtered manifolds \cite{R99,R01,R05,FT23} has been proposed in \cite{H22}.
	This analytic torsion is only defined if the osculating algebras of the filtered manifold have pure cohomology.
	The latter assumption appears to be rather restrictive \cite[Section~3.7]{H22}.
	We are only aware of three types of filtered manifolds with this property:
	trivially filtered manifolds giving rise to the Ray--Singer torsion, contact manifolds giving rise the Rumin--Seshadri torsion, and generic rank two distributions in dimension five which are also known as (2,3,5) distributions \cite{H22}.

	A generic rank two distributions in dimension five is a rank two subbundle $\mathcal D$ in the tangent bundle of a 5-manifold $M$ such that Lie brackets of sections of $\mathcal D$ span a subbundle $[\mathcal D,\mathcal D]$ of rank three, and triple brackets of sections of $\mathcal D$ span all of the tangent bundle, $[\mathcal D,[\mathcal D,\mathcal D]]=TM$.
	These geometric structures have first been studied by Cartan \cite{C10}.
	The Lie bracket of vector fields induces a fiberwise Lie bracket on the associated graded bundle,
	\[
		\mathfrak tM=\tfrac{TM}{[\mathcal D,\mathcal D]}\oplus\tfrac{[\mathcal D,\mathcal D]}{\mathcal D}\oplus\mathcal D.
	\]
	This is a locally trivial bundle of graded Lie algebras over $M$ called the bundle of osculating algebras.
	Its fibers are all isomorphic to the 5-dimensional graded nilpotent Lie algebra 
	\[
		\goe=\goe_{-3}\oplus\goe_{-2}\oplus\goe_{-1}
	\] 
	with graded basis $X_1,X_2\in\goe_{-1}$, $X_3\in\goe_{-2}$, $X_4,X_5\in\goe_{-3}$ and brackets 
	\begin{equation}\label{E:brackets}
		[X_1,X_2]=X_3,\quad[X_1,X_3]=X_4,\quad[X_2,X_3]=X_5.
	\end{equation}

	The simply connected Lie group with Lie algebra $\goe$ will be denoted by $G$.
	The left invariant 2-plane field spanned by $X_1$ and $X_2$ provides a basic example of a (2,3,5) distribution on $G$.
	Unlike contact or Engel structures, (2,3,5) distributions do have local geometry.
	A distribution of this type is locally diffeomorphic to the aforementioned left invariant distribution on $G$ if and only if Cartan's \cite{C10} harmonic curvature tensor, a section of $S^4\mathcal D^*$, vanishes.
	This curvature tensor is constructed using an equivalent description of (2,3,5) distributions as regular normal parabolic geometries of type $(G_2,P)$ where $G_2$ denotes the split real form of the exceptional Lie group and $P$ denotes the parabolic subgroup corresponding to the longer root, see \cite{C10,S08} or \cite[Section~4.3.2]{CS09}.
	Generic rank two distributions in dimension five have attracted quite some attention recently, cf.~\cite{AN14,AN18,BHN18,BM09,BH93,CN09,CSa09,DH19,GPW17,H22,H23,HS09,HS11,LNS17,N05,S06,S08,SW17,SW17b}.

	The Rumin complex associated to a (2,3,5) distribution on a 5-manifold $M$ is a natural sequence of higher order differential operators
	\[
		\cdots\to\Gamma\bigl(\mathcal H^q(\mathfrak tM)\bigr)\xrightarrow{D_q}\Gamma\bigl(\mathcal H^{q+1}(\mathfrak tM)\bigr)\to\cdots
	\]
	where $\mathcal H^q(\mathfrak tM)$ denotes the vector bundle obtained by passing to the fiberwise Lie algebra cohomology of $\mathfrak tM$ with trivial coefficients.
	The Betti numbers are $\rk\mathcal H^q(\mathfrak tM))=\dim H^q(\goe)=1,2,3,2,1$ for $q=0,\dotsc,4$ and the Heisenberg order of the Rumin differential $D_q$ is $k_q=1,3,2,3,1$ for $q=0,\dotsc,4$, see \cite[Section~5]{BENG11} and \cite[Example~4.21]{DH22}.
	The Rumin differentials form a complex, $D_{q+1}D_q=0$, that computes the de~Rham cohomology of $M$.
	Actually, there exist injective differential operators $L_q\colon\Gamma(\mathcal H^q(\mathfrak tM))\to\Omega^q(M)$, embedding the Rumin complex as a subcomplex in the de~Rham complex and inducing  isomorphisms on cohomology.
	Twisting with a flat complex vector bundle $F$ we obtain a complex of differential operators
	\begin{equation}\label{E:Rumin.M}
		\cdots\to\Gamma\bigl(\mathcal H^q(\mathfrak tM)\otimes F\bigr)\xrightarrow{D_q}\Gamma\bigl(\mathcal H^{q+1}(\mathfrak tM)\otimes F\bigr)\to\cdots
	\end{equation}
	computing $H^*(M;F)$, the de~Rham cohomology of $M$ with coefficients in $F$.
	Rumin has shown that the sequence \eqref{E:Rumin.M} becomes exact on the level of the Heisenberg principal symbol, see \cite[Theorem~3]{R99}, \cite[Theorem~5.2]{R01}, or \cite[Corollary~4.18]{DH22}.
	Hence, the Rumin complex is a Rockland \cite{R78} complex, the analogue of an elliptic complex in the Heisenberg calculus, see \cite[Section~2.3]{DH22} for more details.

	A fiberwise graded Euclidean inner product $g$ on $\mathfrak tM$ and a fiberwise Hermitian inner product $h$ on $F$ give rise to $L^2$ inner products on $\Gamma(\mathcal H^q(\mathfrak tM)\otimes F)$ which in turn provide formal adjoints $D_q^*$ of the Rumin differentials in \eqref{E:Rumin.M}.
	Assuming $M$ to be closed, the operator $D_q^*D_q$ has an infinite dimensional kernel if $q>0$ but the remaining part of its spectrum consists of isolated positive eigenvalues with finite multiplicities only.
	Moreover, $(D_q^*D_q)^{-s}$ is trace class for $\Re s>10/2k_q$.
	The number ten appears here because this is the homogeneous dimension of the filtered manifold $M$.
	Furthermore, the function $\tr(D_q^*D_q)^{-s}$ admits an analytic continuation to a meromorphic function on the entire complex plane which is holomorphic at $s=0$, see \cite[Remark~2.9]{H22}.
	This permits to define the zeta regularized determinant $\dets|D_q|$ by
	\begin{equation}\label{E:edtDq.def}
		\log\dets|D_q|:=-\tfrac12\tfrac\partial{\partial s}\big|_{s=0}\tr(D_q^*D_q)^{-s}.
	\end{equation}
	The notation $\dets$ indicates that the zero eigenspace does not contribute, i.e., we are considering the regularized product of nonzero eigenvalues.
	Correspondingly, the complex powers are defined to vanish on the kernel of $D_q^*D_q$.
	The analytic torsion $\tau(M,\mathcal D,F,g,h)$ is the graded determinant of the Rumin complex, i.e.,
	\begin{equation}\label{E:tau.def}
		\log\tau(M,\mathcal D,F,g,h):=\sum_{q=0}^4(-1)^q\log\dets|D_q|.
	\end{equation}
	By Poincar\'e duality, see \cite[Section~2]{R99}, \cite[Proposition~2.8]{R01}, or \cite[Section~3.3]{H22}, we have $\dets|D_0|=\dets|D_4|$ and $\dets|D_1|=\dets|D_3|$, provided $h$ is parallel.
	Hence, in this (unitary) case the torsion may be expressed in terms of three determinants,
	\[
		\tau(M,\mathcal D,F,g,h)=\frac{\dets^2|D_0|\cdot\dets|D_2|}{\dets^2|D_1|}.
	\]

	Since $\goe$ has pure cohomology, the grading automorphism $\phi_\tau\in\Aut_\grr(\goe)$ acts as a scalar on each cohomology, namely by $\tau^{N_q}$ on $H^q(\goe)$ where $N_q=0,1,4,6,9,10$ for $q=0,\dotsc,5$.
	These numbers are related to the Heisenberg orders of the Rumin differentials via the equation $k_q=N_{q+1}-N_q$.
	We fix natural numbers $a_q$ such that $\kappa:=a_qk_q$ is independent of $q$.
	The smallest possible choice would be $a_q=6,2,3,2,6$ with $\kappa=6$.
	In general, $\kappa$ is a multiple of six.
	Then the Rumin--Seshadri operators
	\begin{equation}\label{E:Delta.def}
		\Delta_q:=\bigl(D_{q-1}D_{q-1}^*\bigr)^{a_{q-1}}+\bigl(D_q^*D_q\bigr)^{a_q}
	\end{equation} 
	are all of Heisenberg order $2\kappa$.
	These operators are analytically much better behaved than $D_q^*D_q$.
	Indeed, $\Delta_q$ is a Rockland \cite{R78} operator and, thus, admits a parametrix in the Heisenberg calculus \cite{M82,EY19,DH22}.
	Moreover, $\Delta_q^{-s}$ is trace class for $\Re s>10/2\kappa$, and the zeta function $\tr\Delta_q^{-s}$ admits an analytic continuation to a meromorphic function on the entire complex plane which is holomorphic at $s=0$, see \cite[Corollary~2]{DH20}.
	The analytic properties of the operator $D_q^*D_q$ stated in the preceding paragraph can readily be deduced from the corresponding properties of $\Delta_q$, see \cite[Remark~2.9]{H22} for more details.
	Putting
	\begin{equation}\label{E:zeta.M.def}
		\zeta_{M,\mathcal D,F,g,h}(s):=\str(N\Delta^{-s}):=\sum_{q=0}^5(-1)^qN_q\tr(\Delta_q^{-s})
	\end{equation}
	the analytic torsion of the Rumin complex can be expressed in the form
	\begin{equation}\label{E:tau.via.zeta}
		\tau(M,\mathcal D,F,g,h)=\exp\left(\tfrac1{2\kappa}\zeta'_{M,\mathcal D,F,g,h}(0)\right)
	\end{equation}
	which is analogous to the formulas for the Ray--Singer torsion in \cite[Definition~1.6]{RS71} and the Rumin--Seshadri torsion in \cite[p.~728]{RS12}, see \cite[Eq.~(34)]{H22} for more details.

	It turns out to be convenient to incorporate the zero eigenspaces of $\Delta_q$ and consider the analytic torsion of the Rumin complex as a norm $\|-\|^{\sdet H^*(M;F)}_{\mathcal D,g,h}$ on the graded determinant line
	\[
		\sdet H^*(M;F)=\bigotimes_{q=0}^5\bigl(\det H^q(M;F)\bigr)^{(-1)^q}.
	\]
	Basic properties of this torsion have been established in \cite{H22}.
	In the acyclic case, it reduces to (the reciprocal of) $\tau(M,\mathcal D,F,g,h)$.
	The Ray--Singer torsion, too, is best regarded as a norm on this graded determinant line \cite{BGS88,BZ92}.
	We will denote it by $\|-\|_\RS^{\sdet H^*(M;F)}$.
	The Ray--Singer torsion does not depend on metric choices since the dimension of $M$ is odd.

	In this paper we will determine the analytic torsion of the Rumin complex on nilmanifolds $\Gamma\setminus G$ where $\Gamma$ is a lattice in $G$, i.e., a cocompact discrete subgroup in $G$.
	We consider any left invariant (2,3,5) distribution $\mathcal D_G$ on $G$ and any left invariant fiberwise graded Euclidean inner product on $\mathfrak tG$.
	These descend to a (2,3,5) distribution $\mathcal D_{\Gamma\setminus G}$ on the nilmanifold $\Gamma\setminus G$ and a fiberwise graded Euclidean inner product $g_{\Gamma\setminus G}$ on $\mathfrak t(\Gamma\setminus G)$.
	For a unitary representation $\rho\colon\Gamma\to U(k)$ we let $h_\rho$ denote the canonical (parallel) fiberwise Hermitian inner product on the associated flat complex vector bundle $F_\rho:=G\times_\rho\C^k$ over $\Gamma\setminus G$.

	We are now in a position to state our main result.

	\begin{theorem*}
		If $\chi\colon\Gamma\to U(1)$ is a nontrivial unitary character, then the twisted Rumin complex on the nilmanifold $\Gamma\setminus G$ associated with $\mathcal D_{\Gamma\setminus G}$ and $F_\chi$ is acyclic and its analytic torsion is trivial, that is,
		\[
			\tau\bigl(\Gamma\setminus G,\mathcal D_{\Gamma\setminus G},F_\chi,g_{\Gamma\setminus G},h_\chi\bigr)=1.
		\]
	\end{theorem*}

	In the situation of this theorem, the Ray--Singer torsion is known to be trivial too, $\tau_\RS(\Gamma\setminus G;F_\chi)=1$, see Lemma~\ref{L:RS.Tor} below. 
	Hence, it coincides with the analytic torsion of the Rumin complex.
	More generally, we have:

	\begin{corollary*}
		For any unitary representation $\rho\colon\Gamma\to U(k)$ the analytic torsion of the Rumin complex coincides with the Ray--Singer torsion, that is,
		\[
			\|-\|_{\mathcal D_{\Gamma\setminus G},g_{\Gamma\setminus G},h_\rho}^{\sdet H^*(\Gamma\setminus G;F_\rho)}
			=
			\|-\|_\RS^{\sdet H^*(\Gamma\setminus G;F_\rho)}.
		\]
	\end{corollary*}

	These are the first (2,3,5) distributions for which the analytic torsion of the Rumin complex has been computed.
	In \cite[Theorem~1.3]{H22} a partial result of this type has been obtained by exploiting a large discrete symmetry group of the distribution $\mathcal D_{\Gamma\setminus G}$.

	The proof given below is based on the decomposition of the Rumin complex over $\Gamma\setminus G$ into a countable direct sum $D_*=\bigoplus_\rho m(\rho)\cdot\rho(D_*)$ of Rumin complexes in irreducible unitary representations $\rho$ of $G$, denoted by $\rho(D_*)$.
	The multiplicities $m(\rho)$ of the contributing representations are known explicitly through a formula due to Howe \cite{H71} and Richardson \cite{R71} and involve counting the number of solutions to certain quadratic congruences, cf.~Lemma~\ref{L:Richardson:235}(III), \eqref{E:mult}, or \eqref{E:mdn} below.
	The zeta function in \eqref{E:zeta.M.def} decomposes accordingly, 
	\begin{equation}\label{E:zzz}
		\zeta_{\Gamma\setminus G,\mathcal D_{\Gamma\setminus G},F_\chi,g_{\Gamma\setminus G},h_\chi}(s)=\sum_\rho m(\rho)\cdot\zeta_\rho(s)
	\end{equation}
	where $\zeta_\rho(s)$ denotes the zeta function associated with $\rho(D_*)$.
	The values $\zeta_\rho(0)$ and $\zeta'_\rho(0)$ have been determined in \cite{H23} for every irreducible unitary representation $\rho$.
	Even though the sum on the right hand side in \eqref{E:zzz} converges only for $\Re s>10/2\kappa$, we are able to conclude $\zeta_{\Gamma\setminus G,\mathcal D_{\Gamma\setminus G},F_\chi,g_{\Gamma\setminus G},h_\chi}'(0)=0$ via regularization, whence the theorem stated above.
	This is the subtlest part of the paper at hand and builds on further properties of $\zeta_\rho(s)$ obtained in \cite{H23}.
	Decomposing 
	\[
		\sum_\rho m(\rho)\cdot\zeta_\rho(s)=\zeta_{\I,\Gamma,\chi}(s)+\zeta_{\II,\Gamma,\chi}(s)+\zeta_{\III,\Gamma,\chi}(s)
	\]
	according to the three types of irreducible unitary representations of $G$, we will use Epstein zeta functions to obtain the analytic continuation of each of the three summands, cf.~\eqref{E:factor.I}, \eqref{E:factor.II}, and \eqref{E:factor.III} below.

	Alternatively, one could try to extend Albin and Quan's \cite{AQ22} analysis of the sub-Riemannian limit in order to show that the torsion of the Rumin complex of a (2,3,5) distribution differs from the Ray--Singer torsion by the integral of a local quantity cf.~\cite[Corollary~3]{AQ22}.
	This, too, would immediately imply the corollary stated above.

	The remaining part of this paper is organized as follows.
	In Section~\ref{S:lat} we provide an explicit description of all lattices in $G$.
	In Section~\ref{S:rep.deco} we use the aforementioned result due to Howe \cite{H71} and Richardson \cite{R71} to decompose the space of sections of $F_\chi$ into irreducible $G$-representations.
	In Section~\ref{S:zeta.deco} we describe the corresponding decomposition of the zeta function $\zeta_{\Gamma\setminus G,\mathcal D_{\Gamma\setminus G},F_\chi,g_{\Gamma\setminus G},h_\chi}(s)$.
	In Section~\ref{S:zeta.eval}, building on results obtained in \cite{H23}, we evaluate the derivative of this zeta function at $s=0$.
	In Section~\ref{S:proof} we derive the theorem and corollary stated above.

\section{Lattices}\label{S:lat}

	In this section we provide explicit descriptions of all lattices in $G$.

	Let $X_1,\dotsc,X_5$ be a graded basis of $\goe$ with brackets as in \eqref{E:brackets}.
	For the sake of notational simplicity we will use this basis to identify $\goe$ with $\R^5$, that is, a vector $(x_1,\dotsc,x_5)^t\in\R^5$ will be identified with $\sum_{i=1}^5x_iX_i\in\goe$.

	The exponential map provides a diffeomorphism $\exp\colon\goe\to G$.
	Using the Baker--Campbell--Hausdorff formula we find
	\[
		\exp\begin{pmatrix}x_1\\\vdots\\x_5\end{pmatrix}
		\exp\begin{pmatrix}y_1\\\vdots\\y_5\end{pmatrix}
		=\exp\begin{pmatrix}z_1\\\vdots\\z_5\end{pmatrix}
	\]
	where
	\begin{equation}\label{E:zxy}
		z=x\cdot y
		:=\left(\begin{array}{l}
			x_1+y_1\\
			x_2+y_2\\
			x_3+y_3+\frac{x_1y_2-x_2y_1}2\\x_4+y_4+\frac{x_1y_3-x_3y_1}2+\frac{(x_1-y_1)(x_1y_2-x_2y_1)}{12}\\
			x_5+y_5+\frac{x_2y_3-x_3y_2}2+\frac{(x_2-y_2)(x_1y_2-x_2y_1)}{12}
		\end{array}\right).
	\end{equation}

	The center of $G$ will be denoted by $Z$.
	Clearly, $Z=\exp(\zoe)$ where $\zoe$ denotes the center of $\goe$ which is spanned by $X_4,X_5$.
	For the group of commutators we have $[G,G]=\exp([\goe,\goe])$ and the derived subalgebra $[\goe,\goe]$ is spanned by $X_3,X_4,X_5$.

	For $r\in\N$ and $e,f,g,h,u,v\in\Q$ we consider $\tilde\gamma_i\in\goe$ defined by
	\begin{equation}\label{E:Gamma.tgen}
		\tilde\gamma_1=\begin{pmatrix}1\\0\\0\\0\\0\end{pmatrix},\quad
		\tilde\gamma_2=\begin{pmatrix}0\\1\\0\\0\\0\end{pmatrix},\quad
		\tilde\gamma_3=\begin{pmatrix}0\\0\\1/r\\u/2r\\v/2r\end{pmatrix},\quad
		\tilde\gamma_4=\begin{pmatrix}0\\0\\0\\e\\f\end{pmatrix},\quad
		\tilde\gamma_5=\begin{pmatrix}0\\0\\0\\g\\h\end{pmatrix},
	\end{equation}
	and let $\Gamma$ denote the subgroup of $G$ generated by the exponentials 
	\[
		\gamma_i:=\exp\tilde\gamma_i,\qquad i=1,\dotsc,5.
	\]
	Let $\Gamma''$ denote the lattice in $\R^2$ generated by 
	\begin{equation}\label{E:Gamma''}
		\begin{pmatrix}1/r\\0\end{pmatrix},\quad
		\begin{pmatrix}0\\1/r\end{pmatrix},\quad
		\begin{pmatrix}\frac{u-1}2\\\frac{v-1}2\end{pmatrix},\quad
		\begin{pmatrix}e\\f\end{pmatrix},\quad
		\begin{pmatrix}g\\h\end{pmatrix}.
	\end{equation}

	The following generalizes \cite[Lemma~4.4]{H22}.

	\begin{lemma}\label{L:Gamma}
		The subgroup $\Gamma$ is a lattice in $G$.
		Moreover:
		\begin{align}
		\label{E:Gamma.equ}
		        \log\Gamma
			&=\left\{\begin{pmatrix}x_1\\\vdots\\ x_5\end{pmatrix}\in\goe\,\middle|
		        \begin{array}{rl}
				x_1,x_2&\in\mathbb Z\\
				x_3-\frac{x_1x_2}2&\in\frac1r\Z\\
				\begin{pmatrix}x_4-\tfrac{x_1^2x_2}{12}-\frac{x_1+u}2(x_3-\frac{x_1x_2}2)\\x_5+\tfrac{x_1x_2^2}{12}+\frac{x_2-v}2(x_3-\frac{x_1x_2}2)\end{pmatrix}&\in\Gamma''
		        \end{array}\right\}
			\\
			\label{E:GammaGG}
			\log\bigl(\Gamma\cap[G,G]\bigr)
			&=\left\{\begin{pmatrix}x_1\\\vdots\\ x_5\end{pmatrix}\in\goe\,\middle|
			\begin{array}{rcl}
				x_1,x_2&=&0\\
				x_3&\in&\frac1r\Z\\
				\begin{pmatrix}x_4-\frac u2x_3\\x_5-\frac v2x_3\end{pmatrix}&\in&\Gamma''
		        \end{array}\right\}
			\\
			\label{E:GammaZ}
			\log\bigl(\Gamma\cap Z\bigr)
			&=\left\{\begin{pmatrix}x_1\\\vdots\\ x_5\end{pmatrix}\in\goe\,\middle|
			\begin{array}{rcl}
				x_1,x_2,x_3&=&0\\
				\begin{pmatrix}x_4\\x_5\end{pmatrix}&\in&\Gamma''
		        \end{array}\right\}
			\\
			\label{E:GammaGamma}
			\log\bigl([\Gamma,\Gamma]\bigr)
			&=\left\{\begin{pmatrix}x_1\\\vdots\\ x_5\end{pmatrix}\in\goe\,\middle|
			\begin{array}{rcl}
				x_1,x_2&=&0\\
				x_3&\in&\Z\\
				x_4-\frac{x_3}2,x_5-\frac{x_3}2&\in&\frac1r\Z
		        \end{array}\right\}
			\\
			\label{E:GammaGammaZ}
			\log\bigl([\Gamma,\Gamma]\cap Z\bigr)
			&=\left\{\begin{pmatrix}x_1\\\vdots\\ x_5\end{pmatrix}\in\goe\,\middle|
			\begin{array}{rcl}
				x_1,x_2,x_3&=&0\\
				x_4,x_5&\in&\frac1r\Z
		        \end{array}\right\}
		\end{align}
	\end{lemma}

	\begin{proof}
		With $z$ as in \eqref{E:zxy} we have
		\begin{align*}
			z_3-\tfrac{z_1z_2}2&=x_3-\tfrac{x_1x_2}2+y_3-\tfrac{y_1y_2}2-x_2y_1\\
			z_4-\tfrac{z_1^2z_2}{12}-\tfrac{z_1+u}2(z_3-\tfrac{z_1z_2}2)
			&=x_4-\tfrac{x_1^2x_2}{12}-\tfrac{x_1+u}2(x_3-\tfrac{x_1x_2}2)\\
			&\qquad+y_4-\tfrac{y_1^2y_2}{12}-\tfrac{y_1+u}2(y_3-\tfrac{y_1y_2}2)\\
			&\qquad-y_1(x_3-\tfrac{x_1x_2}2)+\tfrac{u-1}2x_2y_1+\tfrac{y_1(y_1+1)x_2}2\\
			z_5+\tfrac{z_1z_2^2}{12}+\tfrac{z_2-v}2(z_3-\tfrac{z_1z_2}2)
			&=x_5+\tfrac{x_1x_2^2}{12}+\tfrac{x_2-v}2(x_3-\tfrac{x_1x_2}2)\\
			&\qquad+y_5+\tfrac{y_1y_2^2}{12}+\tfrac{y_2-v}2(y_3-\tfrac{y_1y_2}2)\\
			&\qquad+x_2(y_3-\tfrac{y_1y_2}2)+\tfrac{v-1}2x_2y_1-\tfrac{x_2(x_2+1)y_1}2
		\end{align*}
		as well as:
		\begin{align*}
			(-x_3)-\tfrac{(-x_1)(-x_2)}2&=-\bigl(x_3-\tfrac{x_1x_2}2\bigr)-x_1x_2\\
			(-x_4)-\tfrac{(-x_1)^2(-x_2)}{12}-\tfrac{(-x_1)+u}2&\bigl((-x_3)-\tfrac{(-x_1)(-x_2)}2\bigr)\\
			&=-\Bigl(x_4-\tfrac{x_1^2x_2}{12}-\tfrac{x_1+u}2(x_3-\tfrac{x_1x_2}2)\Bigr)\\
			&\qquad-x_1\bigl(x_3-\tfrac{x_1x_2}2\bigr)+\tfrac{u-1}2x_1x_2-\tfrac{x_1(x_1-1)x_2}2\\
			(-x_5)+\tfrac{(-x_1)(-x_2)^2}{12}+\tfrac{(-x_2)-v}2&\bigl((-x_3)-\tfrac{(-x_1)(-x_2)}2\bigr)\\
			&=-\Bigl(x_5+\tfrac{x_1x_2^2}{12}+\tfrac{x_2-v}2(x_3-\tfrac{x_1x_2}2)\Bigr)\\
			&\qquad+x_2\bigl(x_3-\tfrac{x_1x_2}2\bigr)+\tfrac{v-1}2x_1x_2+\tfrac{x_1x_2(x_2+1)}2\\
		\end{align*}
		Using these relations one readily shows that the right hand side in \eqref{E:Gamma.equ} is a lattice containing $\Gamma$.
		Using the computations
		\begin{equation}\label{E:g312}
			\log(\gamma_1^k\gamma_2^l)=\begin{pmatrix}k\\l\\kl/2\\k^2l/12\\-kl^2/12\end{pmatrix},\quad
			\log\bigl(\gamma_3^r[\gamma_1,\gamma_2]^{-1}\bigr)=\begin{pmatrix}0\\0\\0\\\frac{u-1}2\\\frac{v-1}2\end{pmatrix},
		\end{equation}
		\begin{equation}\label{E:g12.13}
			\log[\gamma_1,\gamma_3]=\begin{pmatrix}0\\0\\0\\1/r\\0\end{pmatrix},\quad
			\log[\gamma_2,\gamma_3]=\begin{pmatrix}0\\0\\0\\0\\1/r\end{pmatrix},
		\end{equation}
		it is easy to see that this lattice is generated by $\gamma_1,\dotsc,\gamma_5$.
		Taking also into account
		\begin{equation}\label{E:Gamma.comm}
			\log[\gamma_1,\gamma_2]=\begin{pmatrix}0\\0\\1\\1/2\\1/2\end{pmatrix},\quad
			\log[\gamma_1,[\gamma_1,\gamma_2]]=\begin{pmatrix}0\\0\\0\\1\\0\end{pmatrix},\quad
			\log[\gamma_2,[\gamma_1,\gamma_2]]=\begin{pmatrix}0\\0\\0\\0\\1\end{pmatrix},
		\end{equation}
		we obtain the description of $[\Gamma,\Gamma]$ in \eqref{E:GammaGamma}.
		Here the formula
		\[
			\log[\exp x,\exp y]=\begin{pmatrix}0\\0\\x_1y_2-x_2y_1\\x_1y_3-x_3y_1+\frac{(x_1+y_1)(x_1y_2-x_2y_1)}2\\x_2y_3-x_3y_2+\frac{(x_2+y_2)(x_1y_2-x_2y_1)}2\end{pmatrix}
		\]
		for commutators is helpful.
		The remaining assertions are now obvious.
	\end{proof}

	\begin{lemma}\label{L:all.lat}
		Every lattice in $G$ is of the form considered above, up to a not necessarily graded automorphism of $G$.
	\end{lemma}

	\begin{proof}
		Inspecting the proof of \cite[Theorem~2.21]{R72} we see that every lattice $\Gamma$ in $G$ is generated by five elements $\gamma_1,\dotsc,\gamma_5$ such that $\gamma_1,\gamma_2\in\Gamma$, $\gamma_3\in\Gamma\cap[G,G]$, and $\gamma_4,\gamma_5\in\Gamma\cap Z$. 
		By \cite[Lemma~4.1]{H22} there exists a not necessarily graded automorphism of $\goe$ that maps $\log\gamma_1$ to $\tilde\gamma_1$ and $\log\gamma_2$ to $\tilde\gamma_2$, cf.~\eqref{E:Gamma.tgen}.
		Up to an automorphism of $G$ we may thus assume $\gamma_1=\exp\tilde\gamma_1$ and $\gamma_2=\exp\tilde\gamma_2$.
		As $\gamma_3\in[G,G]$ and $\gamma_4,\gamma_5\in Z$, they must be of the form $\gamma_i=\exp\tilde\gamma_i$ where $\tilde\gamma_i$, $i=3,4,5$, are as indicated in \eqref{E:Gamma.tgen} with, a priori, real numbers $r,u,v,e,f,g,h$.
		To see that these numbers must all be rational it suffices to observe that $\log(\Gamma\cap Z)$ is a lattice in $\goe_{-3}=\R^2$ which contains the vectors in \eqref{E:Gamma''} by \eqref{E:g312} and \eqref{E:g12.13}, but also contains the two unit vectors in view of \eqref{E:Gamma.comm}.
		Write $1/r=s/t$ where $s$ and $t\geq1$ are coprime integers.
		Hence, there exist integers $k$ and $l$ such that $ks+lt=1$.
		Note that $\gamma_1,\gamma_2,\gamma_3^k[\gamma_1,\gamma_2]^l,\gamma_4,\gamma_5$ still generate $\Gamma$ for we have $\bigl(\gamma_3^k[\gamma_1,\gamma_2]^l\bigr)^s=\gamma_3\mod Z$ in view of \eqref{E:g312}, and we may assume that $\gamma_4,\gamma_5$ generate $\Gamma\cap Z$.
		Replacing $\gamma_3$ with $\gamma_3^k[\gamma_1,\gamma_2]^l$, we may thus assume that the image of $\log\gamma_3$ in $[\goe,\goe]/\zoe=\goe_{-2}=\R$ is $1/t$.
	\end{proof}

	The natural homomorphism $p\colon G\to G/[G,G]\cong\R^2$ gives rise to a short exact sequence of abelian groups
	\begin{equation}\label{E:abelGamma}
		0\to\frac{\Gamma\cap[G,G]}{[\Gamma,\Gamma]}\to\frac\Gamma{[\Gamma,\Gamma]}\to p(\Gamma)\to0
	\end{equation}
	where $\frac{\Gamma\cap[G,G]}{[\Gamma,\Gamma]}$ is a finite abelian group and $p(\Gamma)\cong\Z^2$.
	In particular, the sequence splits and we obtain an isomorphism
	\begin{equation}\label{E:Gamma.ab}
		\frac\Gamma{[\Gamma,\Gamma]}\cong\frac{\Gamma\cap[G,G]}{[\Gamma,\Gamma]}\oplus\Z^2.
	\end{equation}
	The group of unitary characters of $\Gamma$, thus, is a finite union of 2-tori,
	\begin{equation}\label{E:Gamma*}
		\hom\bigl(\Gamma,U(1)\bigr)\cong A\times U(1)\times U(1),
	\end{equation}
	where $A=\hom\left(\frac{\Gamma\cap[G,G]}{[\Gamma,\Gamma]},U(1)\right)$ is a finite abelian group.

	To specify a character $\chi\colon\Gamma\to U(1)$ it suffices to know its values on the generators, $\chi(\gamma_i)$ for $i=1,\dotsc,5$.

	\begin{lemma}\label{L:lm0}
		Suppose $\chi\colon\Gamma\to U(1)$ is a unitary character, and let $c$ be a real number such that $\chi(\gamma_3)=e^{2\pi\mathbf ic/r}$.
		Then there exist integers $\lambda_0,\mu_0$ such that 
		\begin{equation}\label{E:lm0i}
			\lambda_0\in r\mathbb Z,\quad\mu_0\in r\mathbb Z,\quad\lambda_0\tfrac{u-1}2+\mu_0\tfrac{v-1}2\in c+\mathbb Z,
		\end{equation}
		\begin{equation}\label{E:lm0ii}
			e^{2\pi\mathbf i(\lambda_0e+\mu_0f)}=\chi(\gamma_4),\quad e^{2\pi\mathbf i(\lambda_0g+\mu_0h)}=\chi(\gamma_5).
		\end{equation}
	\end{lemma}

	\begin{proof}
		Since $\Gamma\cap Z$ is a lattice in $Z\cong\R^2$ it admits a basis consisting of two elements.
		Clearly, there exists a functional $\alpha\in\zoe^*$ such that the homomorphism $Z\to U(1)$, $z\mapsto e^{2\pi\mathbf i\alpha(\log z)}$ coincides with $\chi$ on the aforementioned basis of $\Gamma\cap Z$.
		We conclude that 
		\begin{equation}\label{E:chiZ}
			\chi(\gamma)=e^{2\pi\mathbf i\alpha(\log\gamma)}
		\end{equation}
		for all $\gamma\in\Gamma\cap Z$.
		Via the identification $\zoe^*=(\R^2)^*$, we have $\alpha=(\lambda_0,\mu_0)$ for some real numbers $\lambda_0$ and $\mu_0$.
		Putting $\gamma=\gamma_4$ and $\gamma=\gamma_5$ in \eqref{E:chiZ}, we obtain the two equations in \eqref{E:lm0ii}, respectively.
		Putting $\gamma=\gamma_3^r[\gamma_1,\gamma_2]^{-1}$ in \eqref{E:chiZ} and using \eqref{E:g312} we obtain the last equation in \eqref{E:lm0i}, for $\chi(\gamma_3^r[\gamma_1,\gamma_2]^{-1})=e^{2\pi\mathbf ic}$.
		Putting $\gamma=[\gamma_1,\gamma_3]$ and $\gamma=[\gamma_2,\gamma_3]$ in \eqref{E:chiZ} and using \eqref{E:g12.13} we obtain the first and second equation in \eqref{E:lm0i}, respectively, for $\chi$ vanishes on commutators.
	\end{proof}

\section{Decomposition into irreducible representations}\label{S:rep.deco}

	If $\Gamma$ is a lattice in a simply connected nilpotent Lie group $N$, then $L^2(\Gamma\setminus N)$ decomposes into a countable direct sum of irreducible unitary representations of $N$.
	The multiplicities of the representations appearing in this decomposition have been studied by Moore \cite{M65}.
	An explicit formula for these multiplicities has been conjectured by Mostow and proved, independently, by Howe \cite{H71} and Richardson \cite{R71}.
	More generally, for every unitary character $\chi\colon\Gamma\to U(1)$, the induced representation
	\[
		L^2\bigl(N\times_\chi\C\bigr)
		=\left\{g\colon N\to\mathbb C\,
		\middle|\begin{array}{c}\text{$g(\gamma n)=\chi(\gamma)g(n)$ for $\gamma\in\Gamma$ and $n\in N$,}\\|g|\in L^2(\Gamma\setminus N)\end{array}
		\right\}
	\]
	decomposes into a countable direct sum of irreducible unitary representations and the multiplicities are known explicitly, cf.~\cite[Theorem~1]{H71} and \cite[Theorem~5.3]{R71}.

	\begin{lemma}[Howe, Richardson]\label{L:Richardson}
		Let $N$ be a simply connected nilpotent Lie group, suppose $\Gamma$ a lattice in $N$, and let $\chi\colon\Gamma\to U(1)$ denote a unitary character.
		Then $L^2(N\times_\chi\mathbb C)$ decomposes into a countable direct sum of irreducible unitary representations of $N$.
		If an irreducible unitary representation of $N$ appears in this decomposition, then it is induced from a rational maximal character via Kirillov's construction.
		Moreover, if the rational maximal character $(\bar f,M)$ induces $\pi$, then the multiplicity of $\pi$ in $L^2(N\times_\chi\mathbb C)$ coincides with the (finite) cardinality $\sharp((M\setminus N)_\chi/\Gamma)$, that is, the number of $\Gamma$ orbits in 
		\[
			(M\setminus N)_\chi
			=\left\{n\in M\setminus N\,\middle|
			\begin{array}{c}\text{$(\bar f^n,M^n)$ is rational, and}\\\bar f^n|_{\Gamma\cap M^n}=\chi|_{\Gamma\cap M^n}\end{array}
			\right\}.
		\]
		Here $\bar f^n(m)=\bar f(nmn^{-1})$ and $M^n=n^{-1}Mn$ for $n\in N$ and $m\in M$.
	\end{lemma}

	Let us briefly recall some of the terminology used in the preceding lemma.
	To this end let $\noe$ denote the Lie algebra of $N$, suppose $f\in\noe^*$, and let $\moe$ denote a maximal subordinate subalgebra, i.e., a subalgebra of maximal dimension in $\noe$ such that $f([\moe,\moe])=0$.
	Then 
	\begin{equation}\label{E:barf}
		\bar f(m):=e^{2\pi\mathbf if(\log m)}
	\end{equation}
	defines a unitary character on the group $M:=\exp(\moe)$.
	According to Kirillov \cite{K62,K04,CG90} such a character induces an irreducible unitary representation of $N$ given by right translation on the Hilbert space
	\[
		\left\{h\colon N\to\C\,
		\middle|\begin{array}{c}\text{$h(mn)=\bar f(m)h(n)$ for $m\in M$ and $n\in N$,}\\|h|\in L^2(M\setminus N)\end{array}
		\right\}.
	\]
	One may always assume $\moe$ to be special \cite[\S3]{R71} and then $(\bar f,M)$ is called a maximal character \cite[\S4]{R71}.
	The group $N$ acts by conjugation on the set of maximal characters \cite[Lemma~4.1]{R71} and the stabilizer of $(\bar f,M)$ is $M$, see \cite[Lemma~5.1]{R71}.
	A maximal character $(\bar f,M)$ is called rational \cite[\S3]{R71} if it can be obtained from a (rational) linear functional $f\in\noe^*$ mapping $\log\Gamma$ into the rational numbers and $\moe$ is rational with respect to the rational structure on $\noe$ provided by $\log\Gamma$.
	Note that $\moe$ is rational if and only if $\Gamma\cap M$ is a lattice in $M$, i.e., if and only if $(\Gamma\cap M)\setminus M$ is compact.
	Clearly, the action of $\Gamma$ preserves the subset of rational maximal characters and the integrality condition $\bar f|_{\Gamma\cap M}=\chi|_{\Gamma\cap M}$.

	In the remaining part of this section we will specialize the preceding lemma to the 5-dimensional Lie group $G$ considered before.
	All irreducible unitary representations of this Lie group have been described explicitly by Dixmier in \cite[Proposition~8]{D58}. 
	There are three types of irreducible unitary representations of $G$ which we now describe in terms of a graded basis $X_1,\dotsc,X_5$ of $\goe$ satisfying \eqref{E:brackets}.
	If $\rho$ is a representation of $G$, we let $\rho'$ denote the corresponding infinitesimal representation of $\goe$.

	\begin{enumerate}[(I)]
		\item	\emph{Scalar representations:}
			For $(\alpha,\beta)\in\R^2$ there is an irreducible unitary representation $\rho_{\alpha,\beta}$ of $G$ on $\C$ such that
			\begin{equation}\label{E:rep.scalar}
				\rho'_{\alpha,\beta}(X_1)=2\pi\mathbf i\alpha,\qquad
				\rho'_{\alpha,\beta}(X_2)=2\pi\mathbf i\beta,
			\end{equation}
			and $\rho'_{\alpha,\beta}(X_3)=\rho'_{\alpha,\beta}(X_4)=\rho'_{\alpha,\beta}(X_5)=0$.
			Via Kirillov's construction, the functional $f\in\goe^*$ with $f(X_1)=\alpha$, $f(X_2)=\beta$, and $f(X_3)=f(X_4)=f(X_5)=0$ induces a representation isomorphic to $\rho_{\alpha,\beta}$.
			These are precisely the irreducible representations which factor through the abelianization $G/[G,G]\cong\R^2$.
		\item	\emph{Schr\"odinger representations:}
			For $0\neq\hbar\in\R$ there is an irreducible unitary representation $\rho_\hbar$ of $G$ on $L^2(\R)=L^2(\R,d\theta)$ such that
			\begin{equation}\label{E:Schroedinger}
				\rho'_\hbar(X_1)=\partial_\theta,\qquad
				\rho'_\hbar(X_2)=2\pi\ii\hbar\cdot\theta,\qquad
				\rho'_\hbar(X_3)=2\pi\ii\hbar,
			\end{equation}
			and $\rho'_\hbar(X_4)=\rho'_\hbar(X_5)=0$.
			Via Kirillov's construction, any functional $f\in\goe^*$ with $f(X_3)=\hbar$ and $f(X_4)=f(X_5)=0$ induces a representation isomorphic to $\rho_\hbar$.
			These are precisely the irreducible unitary representations which factor through the 3-dimensional Heisenberg group $H=G/Z$ but do not factor through the abelianization $G/[G,G]$.
		\item	\emph{Generic representations:}
			For real numbers $\lambda,\mu,\nu$ with $(\lambda,\mu)\neq(0,0)$ there is an irreducible unitary representation $\rho_{\lambda,\mu,\nu}$ of $G$ on $L^2(\R)=L^2(\R,d\theta)$ such that:
			\begin{align}
				\notag
				\rho'_{\lambda,\mu,\nu}(X_1)
				&=\frac\lambda{(\lambda^2+\mu^2)^{1/3}}\cdot\partial_\theta-\frac{2\pi\mathbf i\mu}{(\lambda^2+\mu^2)^{1/3}}\cdot\frac{\theta^2+\nu(\lambda^2+\mu^2)^{-2/3}}2
				\\\notag
				\rho'_{\lambda,\mu,\nu}(X_2)
				&=\frac\mu{(\lambda^2+\mu^2)^{1/3}}\cdot\partial_\theta+\frac{2\pi\mathbf i\lambda}{(\lambda^2+\mu^2)^{1/3}}\cdot\frac{\theta^2+\nu(\lambda^2+\mu^2)^{-2/3}}2
				\\\label{E:rep.gen.X3}
				\rho'_{\lambda,\mu,\nu}(X_3)
				&=2\pi\mathbf i(\lambda^2+\mu^2)^{1/3}\cdot\theta,
				\\\notag
				\rho'_{\lambda,\mu,\nu}(X_4)
				&=2\pi\mathbf i\lambda,
				\\\notag
				\rho'_{\lambda,\mu,\nu}(X_5)
				&=2\pi\mathbf i\mu.
			\end{align}
			This differs from the representation given in \cite[Equation~(24)]{D58} by a conjugation with a unitary scaling on $L^2(\R)$ which we have introduced for better compatibility with the grading automorphism. 
			Note that
			\begin{equation*}
				\rho'_{\lambda,\mu,\nu}\bigl(X_3X_3+2X_1X_5-2X_2X_4\bigr)=(2\pi)^2\nu.
			\end{equation*}
			Any functional $f\in\goe^*$ with $f(X_4)=\lambda$, $f(X_5)=\mu$, and 
			\begin{equation}\label{E:rep.gen.nu}
				f(X_3)^2+2f(X_1)f(X_5)-2f(X_2)f(X_4)=\nu
			\end{equation}
			induces a representation isomorphic to $\rho_{\lambda,\mu,\nu}$.
	\end{enumerate}
	The representations listed above are mutually nonequivalent, and they comprise all equivalence classes of irreducible unitary representations of $G$.
	Given the convention in \eqref{E:barf} it turns out to be convenient to incorporate the factors $2\pi$ in the labeling of the representations, as indicated above, when considering integrality with respect to a lattice $\Gamma$ generated by the exponentials of $\tilde\gamma_i$ as in \eqref{E:Gamma.tgen}.

	\begin{lemma}\label{L:Richardson:235}
		Let $\Gamma$ denote the lattice spanned by $\gamma_1,\dotsc,\gamma_5$, where $\gamma_i=\exp\tilde\gamma_i$ and $\tilde\gamma_1,\dotsc,\tilde\gamma_5$ are as indicated in \eqref{E:Gamma.tgen}.
		Suppose $\chi\colon\Gamma\to U(1)$ is a unitary character, and let $a,b,c$ be real numbers such that $\chi(\gamma_1)=e^{2\pi\mathbf ia}$, $\chi(\gamma_2)=e^{2\pi\mathbf ib}$, and $\chi(\gamma_3)=e^{2\pi\mathbf ic/r}$.
		Then $L^2(G\times_\chi\mathbb C)$ decomposes into a countable direct sum of irreducible unitary representations with the following multiplicities:
		\begin{enumerate}[(I)]
			\item	For $\alpha,\beta\in\mathbb R$ the representation $\rho_{\alpha,\beta}$ appears with multiplicity
				\[	
					m\bigl(\rho_{\alpha,\beta}\bigr)
					=\begin{cases}
						1&\text{if $\chi|_{\Gamma\cap[G,G]}=1$, $\alpha\in a+\mathbb Z$, $\beta\in b+\mathbb Z$, and}
						\\0&\text{otherwise.}
					\end{cases}
				\]
				The condition $\chi|_{\Gamma\cap[G,G]}=1$ is equivalent to $c\in r\Z$ and $\chi(\gamma_4)=\chi(\gamma_5)=1$. 
			\item	For $0\neq\hbar\in\mathbb R$ the representation $\rho_\hbar$ appears with multiplicity
				\[
					m\bigl(\rho_\hbar\bigr)
					=\begin{cases}
						|\hbar|&\text{if $\chi|_{\Gamma\cap Z}=1$, $\hbar\in c+r\Z$, and}
						\\0&\text{otherwise.}
					\end{cases}
				\]
				The condition $\chi|_{\Gamma\cap Z}=1$ is equivalent to $c\in\mathbb Z$ and $\chi(\gamma_4)=\chi(\gamma_5)=1$.
			\item	For $\lambda,\mu,\nu\in\mathbb R$ with $(\lambda,\mu)\neq(0,0)$ the multiplicity of the representation $\rho_{\lambda,\mu,\nu}$ vanishes unless 
				\begin{equation}\label{E:lm.lat1}
					\lambda\in r\mathbb Z,\quad\mu\in r\mathbb Z,\quad\lambda\tfrac{u-1}2+\mu\tfrac{v-1}2\in c+\mathbb Z,
				\end{equation}
				\begin{equation}\label{E:lm.lat2}
					e^{2\pi\mathbf i(\lambda e+\mu f)}=\chi(\gamma_4),\quad e^{2\pi\mathbf i(\lambda g+\mu h)}=\chi(\gamma_5),
				\end{equation}
				and
				\begin{equation}\label{E:lm.lat3}
					\nu=\nu_0\mod r\Z
				\end{equation}
				where
				\begin{equation}\label{E:nu0.def}
					\nu_0:=2(a\mu-b\lambda)+\frac{\lambda^2\mu^2}{12d^2}+\frac{\bigl(2w-(\lambda+\mu)+\lambda\mu/d\bigr)^2}4
				\end{equation}
				with $d:=\gcd(\lambda,\mu)$ and
				\begin{equation}\label{E:w.def}
					w:=c-\lambda\tfrac{u-1}2-\mu\tfrac{v-1}2.
				\end{equation}
				In this case, the multiplicity is
				\begin{equation}\label{E:mlmn}
					m\bigl(\rho_{\lambda,\mu,\nu}\bigr)
					=\sharp\Bigl\{k\in\mathbb Z/\tfrac dr\mathbb Z
					\Bigm|\nu=\nu_0+rk(rk+d)+2rkw\mod 2d\Z
					\Bigr\}.
				\end{equation}
		\end{enumerate}
	\end{lemma}

	\begin{proof}
		We specialize Lemma~\ref{L:Richardson} to $N=G$.
		Suppose $f\in\goe^*$ and put $f_i=f(X_i)$.

		Let us begin by considering the case $f|_{[\goe,\goe]}=0$, i.e., $f_3=f_4=f_5=0$.
		Via Kirillov's construction, such a functional induces a representation isomorphic to the scalar representation labeled $\rho_{f_1,f_2}$ in \eqref{E:rep.scalar}.
		In this case, $\moe=\goe$ is the unique (rational) maximal subordinated subalgebra.
		Hence, $M=G$.
		The corresponding maximal character $(\bar f,M)$ is rational iff $f_1$ and $f_2$ are both rational numbers, cf.~\eqref{E:Gamma.tgen} and \eqref{E:barf}.
		Moreover, $\bar f|_{\Gamma\cap M}=\chi|_{\Gamma\cap M}$ if and only if $\bar f(\gamma_i)=\chi(\gamma_i)$ for $i=1,\dotsc,5$.
		As $\bar f(\gamma_i)=e^{2\pi\mathbf if_i}$ for $i=1,2$ and $\bar f(\gamma_i)=1$ for $i=3,4,5$, this is the case iff $f_1\in a+\mathbb Z$, $f_2\in b+\mathbb Z$, and $1=\chi(\gamma_3)=\chi(\gamma_4)=\chi(\gamma_5)$.
		As $(\bar f,M)$ is fixed under the action of $G$, each of these representations occurs with multiplicity one in view of Lemma~\ref{L:Richardson}.
		According to Lemma~\ref{L:Gamma}, the group $\Gamma\cap[G,G]$ is generated by $\gamma_3,[\gamma_1,\gamma_3],[\gamma_2,\gamma_3],\gamma_3^r[\gamma_1,\gamma_2]^{-1},\gamma_4,\gamma_5$, see \eqref{E:g312} and \eqref{E:g12.13}.
		Hence, the condition $\chi|_{\Gamma\cap[G,G]}=1$ is equivalent to $\chi(\gamma_3)=\chi(\gamma_4)=\chi(\gamma_5)=1$.

		Let us next consider the case $f|_{\zoe}=0$ and $f_{[\goe,\goe]}\neq0$, i.e., $f_4=f_5=0$ and $f_3\neq0$.
		Via Kirillov's construction, such a functional induces a representation isomorphic to the Schr\"odinger representation labeled $\rho_{f_3}$ in \eqref{E:Schroedinger}.
		In this case, the subspace $\moe$ spanned by $X_2,X_3,X_4,X_5$ is a rational maximal subordinated subalgebra which is stable under the action of $G$.
		Using Lemma~\ref{L:Gamma}, we see that the group $\Gamma\cap M$ is generated by 
		\[
			\gamma_2,\quad\gamma_3,\quad\text{and}\quad\Gamma\cap Z.
		\]
		The corresponding maximal character $(\bar f,M)$ is rational iff $f_2$ and $f_3$ are both rational numbers.
		Moreover, $\bar f|_{\Gamma\cap M}=\chi|_{\Gamma\cap M}$ if and only if $\bar f(\gamma_2)=\chi(\gamma_2)$, $\bar f(\gamma_3)=\chi(\gamma_3)$, and $\bar f|_{\Gamma\cap Z}=\chi|_{\Gamma\cap Z}$.
		Equivalently, $f_2\in b+\mathbb Z$, $f_3\in c+\mathbb Z$, and $1=\chi|_{\Gamma\cap Z}$.
		Suppose $g=\exp(\sum_{i=1}^5x_iX_i)$.
		A straightforward calculation yields $(\Ad_g^*f)(X_2)=f_2+x_1f_3$ and $(\Ad_g^*f)(X_i)=f_i$ for $i=3,4,5$.
		Hence, $(\bar f^g,M^g=M)$ is rational and $\bar f^g|_{\Gamma\cap M^g}=\chi|_{\Gamma\cap M^g}$ iff $f_2+x_1f_3\in b+\mathbb Z$, $f_3\in c+\mathbb Z$, and $\chi|_{\Gamma\cap Z}=1$.
		As integral $x_1$ correspond to $g\in\Gamma$, we conclude form Lemma~\ref{L:Richardson}, that the multiplicity is $|f_3|$, provided $f_3\in c+\mathbb Z$, and $\chi|_{\Gamma\cap Z}=1$.
		According to Lemma~\ref{L:Gamma}, the group $\Gamma\cap Z$ is generated by $[\gamma_1,\gamma_3],[\gamma_2,\gamma_3],\gamma_3^r[\gamma_1,\gamma_2]^{-1},\gamma_4,\gamma_5$.
		Hence, the condition $\chi|_{\Gamma\cap Z}=1$ is equivalent to $c\in\mathbb Z$ and $1=\chi(\gamma_4)=\chi(\gamma_5)$.
	
		Let us finally turn to the case $f|_{\zoe}\neq0$, i.e., $(f_4,f_5)\neq(0,0)$.
		Via Kirillov's construction, such a functional induces a representation isomorphic to the generic representation labeled $\rho_{\lambda,\mu,\nu}$ in \eqref{E:rep.gen.X3} where, cf.~\eqref{E:rep.gen.nu}, 
		\begin{equation}\label{E:lmn.v.fi}
			\lambda=f_4,\qquad\mu=f_5,\qquad\nu=f_3^2+2(f_1f_5-f_2f_4).
		\end{equation}
		In this case, the subspace $\moe$ spanned by $f_5X_1-f_4X_2,X_3,X_4,X_5$ is the unique maximal subordinated subalgebra.
		We assume that the corresponding maximal character $(\bar f, M)$ is rational, i.e., $f_1f_5-f_2f_4,f_3,f_4,f_5$ are all rational.
		As the group $\Gamma\cap Z$ is generated by $[\gamma_1,\gamma_3],[\gamma_2,\gamma_3],\gamma_3^r[\gamma_1,\gamma_2]^{-1},\gamma_4,\gamma_5$, we have $\bar f|_{\Gamma\cap Z}=\chi|_{\Gamma\cap Z}$ if and only if \eqref{E:lm.lat1} and \eqref{E:lm.lat2} hold true, cf.~\eqref{E:g312}, \eqref{E:g12.13}, and \eqref{E:Gamma.tgen}.
		We assume from now on that this is the case. 
		In particular, $f_4,f_5$ are integral and we write $d=\gcd(f_4,f_5)\in r\Z$.
		Using Lemma~\ref{L:Gamma} and \eqref{E:zxy} we see that the group $\Gamma\cap M$ is generated by 
		\[
			\gamma_1^{f_5/d}\gamma_2^{-f_4/d}=\exp\begin{pmatrix}f_5/d\\-f_4/d\\-f_4f_5/2d^2\\-f_4f_5^2/12d^3\\-f_4^2f_5/12d^3\end{pmatrix},\quad\gamma_3=\exp\begin{pmatrix}0\\0\\1/r\\u/2r\\v/2r\end{pmatrix},\quad\text{and}\quad\Gamma\cap Z.
		\]
		Hence, $\bar f|_{\Gamma\cap M}=\chi|_{\Gamma\cap M}$ if and only if (furthermore)
		\[
			\frac{f_1f_5-f_2f_4}d-\frac{f_3f_4f_5}{2d^2}-\frac{f_4^2f_5^2}{6d^3}=\frac{af_5-bf_4}d\mod\Z
		\] 
		and 
		\[
			f_3/r+(f_4u+f_5v)/2r=c/r\mod\Z.
		\]
		Using \eqref{E:lmn.v.fi}, \eqref{E:w.def} and \eqref{E:nu0.def}, this is readily seen to be equivalent to
		\[
			\nu=\nu_0+\bigl(f_3+\lambda\mu/2d\bigr)^2-\bigl(w-(\lambda+\mu)/2+\lambda\mu/2d\bigr)^2\mod 2d\Z
		\]
		and
		\[
			f_3=w-(\lambda+\mu)/2\mod r\Z.
		\]
		This is the case if and only if there exists an integer $k$ such that 
		\[
			f_3=rk+w-(\lambda+\mu)/2
		\]
		and
		\[
			\nu=\nu_0+rk\bigl(rk+2w-(\lambda+\mu)+\lambda\mu/d\bigr)\mod2d\Z.
		\]
		The latter can be replaced with the equivalent condition 
		\[
			\nu=\nu_0+rk(rk+d)+2rkw\mod2d\Z,
		\]
		for we have
		\[
			d-(\lambda+\mu)+\lambda\mu/d=d(1-\lambda/d)(1-\mu/d)\in2d\mathbb Z,
		\]
		as $\lambda/d$ and $\mu/d$ can not both be even.
		In particular, \eqref{E:lm.lat3} must hold in this case.
		Recall that $\lambda$, $\mu$ and $\nu$ are invariant under coadjoint action by $G$.
		Moreover, for $g=\exp(x_1X_1+\cdots+x_5X_5)$, we have $(\Ad_g^*f)(X_3)=f_3+x_1f_4+x_2f_5$.
		Hence, $(\bar f,M)$ and $(\bar f^g,M^g)$ lie in the same $\Gamma$-orbit iff $(\Ad_g^*f)(X_3)-f_3$ is divisible by $d$.
		The formula for the multiplicity now follows from Lemma~\ref{L:Richardson}.
	\end{proof}

	For integers $l,r,w,n\in\mathbb Z$ with $l,r\geq1$ we define 
	\begin{equation}\label{E:mult}
		m(l,r,w,n):=\sharp\Bigl\{k\in\mathbb Z/l\mathbb Z\Bigm|rk(k+l)+2wk\equiv n\mod2l\mathbb Z\Bigr\}.
	\end{equation}
	Clearly,
	\begin{equation}\label{E:M.mult}
		m(l,r,w,n+2l)=m(l,r,w,n)
	\end{equation}
	and
	\begin{equation}\label{E:m.total}
		\sum_{n=1}^{2l}m(l,r,w,n)=l.
	\end{equation}

	With this notation, the statement in Lemma~\ref{L:Richardson:235} may be expressed in the form:
	\begin{multline}\label{E:rep.deco}
		L^2(G\times_\chi\mathbb C)
		=\underbrace{\bigoplus_{(\alpha,\beta)\in(a,b)+\Z^2}\rho_{\alpha,\beta}}_{\substack{\textrm{only appears if}\\\textrm{$\chi|_{\Gamma\cap[G,G]}=1$}}}
		\,\,\,\oplus\,\,\,\underbrace{\bigoplus_{\hbar\in c+r\Z}|\hbar|\rho_\hbar}_{\substack{\textrm{only appears if}\\\textrm{$\chi|_{\Gamma\cap Z}=1$}}}
		\\
		\oplus\bigoplus_{\substack{(0,0)\neq(\lambda,\mu)\in\\(\lambda_0,\mu_0)+(\Gamma'')^*}}\quad\bigoplus_{n\in\Z}\,\,\,m\bigl(\tfrac dr,r,w,n\bigr)\rho_{\lambda,\mu,\nu_0+rn}.
	\end{multline}
	Here $\lambda_0,\mu_0\in r\Z$ are as in Lemma~\ref{L:lm0}, and
	\begin{equation}\label{E:Gamma''*}
		(\Gamma'')^*=\left\{(l,m)\in\R^2\middle|\begin{array}{c}\frac lr\in\Z,\frac mr\in\Z,l\frac{u-1}2+m\frac{v-1}2\in\mathbb Z\\le+mf\in\Z,lg+mh\in\Z\end{array}
		\right\}
		\subseteq(r\Z)\times(r\Z)
	\end{equation}
	denotes the lattice dual to the lattice $\Gamma''$ spanned by the vectors in \eqref{E:Gamma''}.
	Moreover, $d=\gcd(\lambda,\mu)$, and $w$, $\nu_0$ are defined \eqref{E:nu0.def} and \eqref{E:w.def}.

\section{Decomposition of the zeta function}\label{S:zeta.deco}

	In this section we use the decomposition in \eqref{E:rep.deco} to decompose the twisted Rumin complex associated with a standard (2,3,5) distribution on the nilmanifold $\Gamma\setminus G$, a standard fiberwise graded Euclidean inner product on $\mathfrak t(\Gamma\setminus G)$, and a unitary character $\chi\colon\Gamma\to U(1)$.
	This yields a decomposition of the corresponding zeta function, cf.~\eqref{E:zzz}.

	We continue to use a graded basis $X_1,\dotsc,X_5$ of $\goe$ satisfying the relations in \eqref{E:brackets}.
	Let $\mathcal D_G$ denote the left invariant distribution on $G$ spanned by $X_1$ and $X_2$.
	Left translation provides a trivialization of the tangent bundle that induces a trivialization of the bundle of osculating algebras, $\mathfrak tG=G\times\goe$.
	Passing to the fiberwise Lie algebra cohomology, we obtain a trivialization of the vector bundle $\mathcal H^q(\mathfrak tG)=G\times H^q(\goe)$.
	Via this identification, the untwisted Rumin differential is a left invariant differential operator $D_q\colon C^\infty(G)\otimes H^q(\goe)\to C^\infty(G)\otimes H^{q+1}(\goe)$ which may be considered as 
	\begin{equation}\label{E:Dq.Ug}
		D_q\in\mathcal U(\goe)\otimes L\bigl(H^q(\goe),H^{q+1}(\goe)\bigr)
	\end{equation}
	where $\mathcal U(\goe)$ denotes the universal enveloping algebra of $\goe$.
	With respect to a particular basis of $H^q(\goe)$, the Rumin differentials take the form:
	\begin{align*}
		D_0&=\begin{pmatrix}X_1\\ X_2\end{pmatrix}
		\\
		D_1&=\begin{pmatrix}-X_{112}-X_{13}-X_4&X_{111}\\-\sqrt2X_{122}-\sqrt2X_5&\sqrt2X_{211}-\sqrt2X_4\\-X_{222}&X_{221}-X_{23}-X_5\end{pmatrix}
		\\
		D_2&=\begin{pmatrix}-X_{12}-X_3&X_{11}/\sqrt2&0\\-X_{22}/\sqrt2&-\tfrac32X_3&X_{11}/\sqrt2\\0&-X_{22}/\sqrt2&X_{21}-X_3\end{pmatrix}
		\\
		D_3&=\begin{pmatrix}X_{122}+X_{32}-X_5&-\sqrt2X_{112}+\sqrt2X_4&X_{111}\\X_{222}&-\sqrt2X_{221}-\sqrt2X_5&X_{211}-X_{31}+X_4\end{pmatrix}
		\\
		D_4&=\begin{pmatrix}-X_2,X_1\end{pmatrix}
	\end{align*}
	Here we are using the notation $X_{j_1\cdots j_k}=X_{j_1}\cdots X_{j_k}$.
	These formulas for $D_q$ have been derived in \cite[Appendix~B]{DH17}, where they are expressed with respect to a slightly different basis, see also \cite[Example~4.21]{DH22} or \cite[Section~3.3]{H23}.

	Suppose $\Gamma$ is a lattice in $G$ that is of the form considered in Section~\ref{S:lat}, i.e., generated by the exponentials of $\tilde\gamma_i$ as in \eqref{E:Gamma.tgen}.
	Furthermore, let $\chi\colon\Gamma\to U(1)$ be a unitary character.
	For the sections of the associated flat line bundle $F_\chi$ over the nilmanifold $\Gamma\setminus G$ we have a canonical identification
	\begin{equation}\label{E:F.chi.sec}
		\Gamma^\infty(F_\chi)
		=C^\infty(G\times_\chi\C)
		:=\left\{f\in C^\infty(G,\C)\middle|\begin{array}{c}f(\gamma g)=\chi(\gamma)f(g)\\\text{for $\gamma\in\Gamma$ and $g\in G$}\end{array}\right\}.
	\end{equation}
	The left invariant 2-plane field $\mathcal D_G$ descends to a generic distribution of rank two on the nilmanifold $\Gamma\setminus G$ which we denote by $\mathcal D_{\Gamma\setminus G}$.
	The trivialization $\mathcal H^q(\mathfrak tG)=G\times H^q(\goe)$ mentioned above gives rise to a trivialization of the vector bundle $\mathcal H^q(\mathfrak t(\Gamma\setminus G))=(\Gamma\setminus G)\times H^q(\goe)$.
	Combining this with \eqref{E:F.chi.sec} we obtain a canonical identification
	\begin{equation}\label{E:HqF.G}
		\Gamma^\infty\bigl(\mathcal H^q(\mathfrak t(\Gamma\setminus G))\otimes F_\chi\bigr)=C^\infty(G\times_\chi\C)\otimes H^q(\goe).
	\end{equation}
	Via this identification the twisted Rumin differentials become operators
	\begin{equation}\label{E:DqG}
		D_q\colon C^\infty(G\times_\chi\C)\otimes H^q(\goe)\to C^\infty(G\times_\chi\C)\otimes H^{q+1}(\goe)
	\end{equation}
	which are given by the matrices in \eqref{E:Dq.Ug}, with the universal algebra now acting in the (induced) representation $C^\infty(G\times_\chi\C)$.

	Let $g$ denote the graded Euclidean inner product on $\goe$ that turns $X_1,\dotsc,X_5$ into an orthonormal basis.
	The corresponding left invariant fiberwise graded Euclidean inner product on the bundle of osculating algebras $\mathfrak tG=G\times\goe$ descends to a fiberwise graded Euclidean inner product $g_{\Gamma\setminus G}$ on the bundle of osculating algebras $\mathfrak t(\Gamma\setminus G)$ over $\Gamma\setminus G$.
	The flat line bundle $F_\chi$ comes with a canonical fiberwise Hermitian inner product denoted by $h_\chi$.
	The metrics $g_{\Gamma\setminus G}$ and $h_\chi$ provide an $L^2$ inner product on the space $\Gamma^\infty\bigl(\mathcal H^q(\mathfrak t(\Gamma\setminus G))\otimes F_\chi\bigr)$.
	Via the identification in \eqref{E:HqF.G} this inner product corresponds to the tensor product of the standard $L^2$ inner product on the induced representation $C^\infty(G\times_\chi\C)$ and the Hermitian inner product on $H^q(\goe)$ induced by $g$.

	In view of \eqref{E:HqF.G}, the decomposition of the induced representation described in \eqref{E:rep.deco} provides a decomposition of the twisted Rumin complex over $\Gamma\setminus G$ into a countable orthogonal direct sum of Rumin complexes, each associated with an irreducible unitary representations of $G$.
	For every irreducible unitary representation $\rho$ of $G$, we let $\rho(D_q)$ denote the operator induced by \eqref{E:Dq.Ug} in this representation, and consider the corresponding Rumin--Seshadri operator,
	\begin{equation}\label{E:Delta.rho}
		\Delta_{\rho,q}:=\bigl(\rho(D_{q-1})\rho(D_{q-1})^*\bigr)^{a_{q-1}}+\bigl(\rho(D_q)^*\rho(D_q)\bigr)^{a_q}.
	\end{equation}
	Moreover, we put
	\begin{equation}\label{E:zeta.rho.def}
		\zeta_\rho(s)
		:=\str\bigl(N\Delta_\rho^{-s}\bigr)
		:=\sum_{q=0}^5(-1)^qN_q\tr\Delta_{\rho,q}^{-s}.
	\end{equation}
	These zeta functions are known to converge for $\Re s$ sufficiently large, depending on $\rho$, and they admit analytic continuation to meromorphic functions on the entire complex plane which are holomorphic at $s=0$, see \cite[Theorem~1]{H23}.
	As the Rumin complex is a Rockland complex, the operators $\Delta_{\rho,q}$ have trivial kernel for nontrivial $\rho$. 
	Clearly, $\Delta_{\rho,q}$ vanishes if $\rho$ is the trivial representation, and so does $\zeta_\rho(s)$.

	Using the notation from Lemma~\ref{L:Richardson:235} we define
	\begin{equation}\label{E:zeta.I}
		\zeta_{\I,\Gamma,\chi}(s)
		:=\sideset{}{'}\sum_{(\alpha,\beta)\in(a,b)+\Z^2}\zeta_{\rho_{\alpha,\beta}}(s)
	\end{equation}
	if $\chi|_{\Gamma\cap[G,G]}=1$, and $\zeta_{\I,\Gamma,\chi}(s):=0$ otherwise.
	Moreover, 
	\begin{equation}\label{E:zeta.II}
		\zeta_{\II,\Gamma,\chi}(s)
		:=\sideset{}{'}\sum_{\hbar\in c+r\Z}|\hbar|\zeta_{\rho_\hbar}(s)
	\end{equation}
	if $\chi|_{\Gamma\cap Z}=1$, and $\zeta_{\II,\Gamma,\chi}(s):=0$ otherwise.
	Finally,
	\begin{equation}\label{E:zeta.III}
		\zeta_{\III,\Gamma,\chi}(s)
		:=\sideset{}{'}\sum_{\substack{(\lambda,\mu)\in\\(\lambda_0,\mu_0)+(\Gamma'')^*}}\,\,\,\sum_{n\in\Z}\,\,\,m(\tfrac dr,r,w,n)\zeta_{\rho_{\lambda,\mu,\nu_0+nr}}(s).
	\end{equation}
	Here $\lambda_0,\mu_0$ are as in Lemma~\ref{L:lm0}, $(\Gamma'')^*$ is the dual lattice in \eqref{E:Gamma''*}, $d=\gcd(\lambda,\mu)$, $w$ is given in \eqref{E:w.def}, $\nu_0$ is given in \eqref{E:nu0.def}, and the multiplicity is defined in \eqref{E:mult}.
	Moreover, we are using the common convention to decorate the summation symbol with a prime to indicate that the summand with index zero (if any) is omitted.

	\begin{lemma}\label{L:zeta.deco}
		For all unitary characters $\chi\colon\Gamma\to U(1)$ and $\Re s>10/2\kappa$ we have
		\begin{equation}\label{E:zeta.deco}
			\zeta_{\Gamma\setminus G,\mathcal D_{\Gamma\setminus G},F_\chi,g_{\Gamma\setminus G},h_\chi}(s)
			=\zeta_{\I,\Gamma,\chi}(s)+\zeta_{\II,\Gamma,\chi}(s)+\zeta_{\III,\Gamma,\chi}(s)
		\end{equation}
		where the left hand side has been defined in \eqref{E:zeta.M.def}.
	\end{lemma}

	\begin{proof}
		Via the identification in \eqref{E:HqF.G} the decomposition in \eqref{E:rep.deco} provides a decomposition of the twisted Rumin complex $D_*$ in \eqref{E:DqG} into a countable orthogonal direct sum of complexes,
		\begin{equation}\label{E:comp.deco}
			D_*
			=\bigoplus_\rho m(\rho)\cdot\rho(D_*),
		\end{equation}
		where the sum is over all irreducible unitary representations $\rho$ of $G$ and the multiplicities are given in Lemma~\ref{L:Richardson:235}.
		By \eqref{E:Delta.def} and \eqref{E:Delta.rho}, and since $\rho(\Delta_q)=\Delta_{\rho,q}$, this yields
		\[
			\Delta_q
			=\bigoplus_\rho m(\rho)\cdot\Delta_{\rho,q}.
		\]
		Using \eqref{E:zeta.M.def} and \eqref{E:zeta.rho.def}, we get
		\[
			\zeta_{\Gamma\setminus G,\mathcal D_{\Gamma\setminus G},F_\chi,g_{\Gamma\setminus G},h_\chi}(s)
			=\sideset{}{'}\sum_\rho m(\rho)\cdot\zeta_\rho(s).
		\]
		Both sides converge for $\Re s>10/2\kappa$ in view of \cite[Corollary~2]{DH20}.
		We omit the trivial representation $\rho$ on the right hand side because $\zeta_\rho(s)$ vanishes identically for this $\rho$.
		Combining this with the description of the multiplicities in Lemma~\ref{L:Richardson:235} and the definitions in \eqref{E:zeta.I} -- \eqref{E:zeta.III}, we obtain the lemma.
	\end{proof}

	Subsequently, we will analyze each of the three summands in \eqref{E:zeta.deco} individually.

	\begin{remark}
		Let us specialize to the simple lattice $\Gamma=\Gamma_0$ generated by $X_1$ and $X_2$.
		In this case $r=1$, $u=v=1$, $e=f=g=h=0$, $c=0$, and $w=0$.
		Hence,
		\[
			\zeta_{\I,\Gamma_0,\chi}(s)
			=\sideset{}{'}\sum_{(\alpha,\beta)\in(a,b)+\Z^2}\zeta_{\rho_{\alpha,\beta}}(s),\qquad
			\zeta_{\II,\Gamma_0,\chi}(s)
			=\sideset{}{'}\sum_{\hbar\in\Z}|\hbar|\zeta_{\rho_\hbar}(s)
		\]
		and
		\[
			\zeta_{\III,\Gamma_0,\chi}(s)
			:=\sideset{}{'}\sum_{(\lambda,\mu)\in\Z^2}\,\,\,\sum_{n\in\Z}\,\,\,m(d,n)\zeta_{\rho_{\lambda,\mu,\nu_0+n}}(s)
		\]
		where $d=\gcd(\lambda,\mu)$,
		\begin{equation}\label{E:mdn}
			m(d,n)
			:=m(d,1,0,n)
			=\sharp\Bigl\{k\in\mathbb Z/d\mathbb Z\Bigm|k(k+d)\equiv n\mod2d\mathbb Z\Bigr\},
		\end{equation}
		and $\nu_0$ as in \eqref{E:nu0.def} with $w=0$.
		For this lattice $\Gamma_0\cap[G,G]=[\Gamma_0,\Gamma_0]$ by Lemma~\ref{L:Gamma}, $\Gamma_0/[\Gamma_0,\Gamma_0]\cong\Z^2$ by \eqref{E:Gamma.ab}, and $\hom(\Gamma_0,U(1))\cong U(1)\times U(1)$ according to \eqref{E:Gamma*}.
	\end{remark}

\section{Evaluation of the zeta function}\label{S:zeta.eval}

	We continue to use the notation set up in the preceding section and Lemma~\ref{L:Richardson:235}.
	In particular, $\Gamma$ denotes the lattice in $G$ spanned by the exponentials of $\tilde\gamma_i$ as in \eqref{E:Gamma.tgen}, $\chi\colon\Gamma\to U(1)$ is a unitary character, and $g$ is a graded Euclidean inner product on $\goe$ such that $X_1,\dotsc,X_5$ is an orthonormal basis of $\goe$.
	This implies, that every graded automorphism $\phi\in\Aut_\grr(\goe)=\GL(\goe_{-1})\cong\GL_2(\R)$ that preserves $g|_{\goe_{-1}}$ also preserves $g$.
	By naturality of the Rumin differentials, $\phi\cdot D_q=D_q$ where the dot denotes the natural left action of $\Aut_\grr(\goe)$ on $\mathcal U(\goe)\otimes L\bigl(H^q(\goe),H^{q+1}(\goe)\bigr)$, cf.~\eqref{E:Dq.Ug}.
	Hence, if $\rho$ is an irreducible unitary representation of $G$, then
	\begin{equation}\label{E:zeta.phi}
		\zeta_{\rho\circ\phi}(s)=\zeta_\rho(s)
	\end{equation}
	for every graded automorphism $\phi\in\Aut_\grr(\goe)$ preserving $g|_{\goe_{-1}}$.
	Moreover, for $\tau>0$ we have
	\begin{equation}\label{E:zeta.phit}
		\zeta_{\rho\circ\phi_\tau}(s)=\tau^{2\kappa s}\zeta_{\rho}(s)
	\end{equation}
	where $\phi_\tau\in\Aut_\grr(\goe)$ denotes the grading automorphism acting by $\tau^j$ on $\goe_j$.
	For more details we refer to \cite[Section~3.1]{H23}.

	\begin{lemma}\label{L:zeta.I}
		If $\chi|_{\Gamma\cap[G,G]}=1$, then the sum in \eqref{E:zeta.I} converges for $\Re s>1/\kappa$, and
		\begin{equation}\label{E:factor.I}
			\zeta_{\I,\Gamma,\chi}(s)
			=Z_\Epst(2\kappa s;a,b)\cdot\zeta_{\rho_{1,0}}(s)
		\end{equation}
		where 
		\[
			Z_\Epst(s;a,b):=\sideset{}{'}\sum_{(\alpha,\beta)\in(a,b)+\Z^2}\bigl(\alpha^2+\beta^2\bigr)^{-s/2}
		\] 
		denotes an Epstein zeta function.
		In particular, $\zeta_{\I,\Gamma,\chi}(s)$ extends to a meromorphic function on the entire complex plane which has a single (simple) pole at $s=1/\kappa$ with residue $\frac\pi\kappa\cdot\zeta_{\rho_{1,0}}(1/\kappa)$.
		Moreover, for all unitary characters $\chi$,
		\begin{equation}\label{E:zeta.I.0}
			\zeta_{\I,\Gamma,\chi}(0)=0,
			\qquad
			\zeta_{\I,\Gamma,\chi}'(0)
			=\begin{cases}0&\text{if $\chi$ is nontrivial, and}\\2\kappa\log2&\text{if $\chi$ is trivial.}\end{cases}
		\end{equation}
	\end{lemma}

	\begin{proof}
		W.l.o.g.\ $\chi|_{\Gamma\cap[G,G]}=1$.
		The homogeneity in \eqref{E:zeta.phit} gives
		\[
			\zeta_{\rho_{\tau\alpha,\tau\beta}}(s)=\tau^{-2\kappa s}\zeta_{\rho_{\alpha,\beta}}(s),\qquad\tau>0.
		\]
		For $(\alpha,\beta)\neq(0,0)$ we may apply \eqref{E:zeta.phi} with a graded automorphism $\phi\in\Aut_\grr(\goe)$ that acts as a rotation on $\goe_{-1}$, and then use the latter homogeneity to obtain
		\[
			\zeta_{\rho_{\alpha,\beta}}(s)=(\alpha^2+\beta^2)^{-\kappa s}\cdot\zeta_{\rho_{1,0}}(s).
		\]
		Summing over all $(\alpha,\beta)$ in the shifted lattice, we obtain \eqref{E:factor.I}, cf.~\eqref{E:zeta.I}.

		As the scalar representations $\rho_{\alpha,\beta}$ are one dimensional, $\zeta_{\rho_{\alpha,\beta}}(s)$ is an entire function, and \eqref{E:zeta.rho.def} yields
		\begin{equation}\label{E:qwerty97}
			\zeta_{\rho_{\alpha,\beta}}(0)=\str(N)=\sum_{q=0}^5(-1)^qN_q\dim H^q(\goe)=0.
		\end{equation}
		By \cite[Theorem~2]{H23},
		\begin{equation}\label{E:qwerty98}
			\exp\left(\tfrac1{2\kappa}\zeta_{\rho_{\alpha,\beta}}'(0)\right)
			=\frac12.
		\end{equation}

		The Epstein \cite[p.~627]{E03} zeta function $Z_\Epst(s;a,b)$ converges for $s>2$ and extends to a meromorphic on the entire complex plane which has a single (simple) pole at $s=2$ with residue $2\pi$, and 
		\begin{equation}\label{E:qwerty99}
			Z_\Epst(0;a,b)=
			\begin{cases}
				-1&\text{if $a,b\in\Z$, and}\\
				0&\text{otherwise.}
			\end{cases}
		\end{equation}
		Note that $a,b\in\Z$ if and only if $\chi$ is trivial, cf.~Lemma~\ref{L:Richardson:235}.
		Combining this with \eqref{E:factor.I} we see that the sum in \eqref{E:zeta.I} converges for $\Re s>1/\kappa$, and that $\zeta_{\I,\Gamma,\chi}(s)$ extends to a meromorphic function on the entire complex plane which has a single (simple) pole at $s=1/\kappa$ of residue $\frac{2\pi}{2\kappa}\cdot\zeta_{\rho_{1,0}}(1/\kappa)$.
		Furthermore, combining \eqref{E:qwerty97}--\eqref{E:qwerty99} with \eqref{E:factor.I} we obtain \eqref{E:zeta.I.0}.
	\end{proof}

	\begin{lemma}\label{L:zeta.II}
		If $\chi|_{\Gamma\cap Z}=1$, then the sum in \eqref{E:zeta.II} converges for $\Re s>2/\kappa$ and
		\begin{equation}\label{E:factor.II}
			\zeta_{\II,\Gamma,\chi}(s)
			=r^{-\kappa s+1}\cdot Z_\Epst(\kappa s-1;\tfrac cr)\cdot\zeta_{\rho_1}(s)
		\end{equation}
		where 
		\[
			Z_\Epst(s;a):=\sideset{}{'}\sum_{n\in\Z}|a+n|^{-s}
		\] 
		denotes an Epstein zeta function.
		In particular, $\zeta_{\II,\Gamma,\chi}(s)$ extends to a meromorphic function on the entire complex plane which has a simple pole at $s=2/\kappa$ with residue $\frac2{r\kappa}\cdot\zeta_{\rho_1}(2/\kappa)$.
		If $c\notin r\Z$, then this is the only pole of $\zeta_{\II,\Gamma,\chi}(s)$.
		If $c\in r\Z$, then $\zeta_{\II,\Gamma,\chi}(s)$ has one further (simple) pole at $s=1/\kappa$ with residue $-\res_{s=1/\kappa}\zeta_{\rho_1}(1/\kappa)$.
		Moreover, for all unitary characters $\chi$,
		\begin{equation}\label{E:zeta.II.0}
			\zeta_{\II,\Gamma,\chi}(0)=0
			\qquad\text{and}\qquad
			\zeta_{\II,\Gamma,\chi}'(0)=0.
		\end{equation}
	\end{lemma}

	\begin{proof}
		W.l.o.g.\ $\chi|_{\Gamma\cap Z}=1$.
		The homogeneity in \eqref{E:zeta.phit} gives
		\[
			\zeta_{\rho_{\tau^2\hbar}}(s)=\tau^{-2\kappa s}\zeta_{\rho_\hbar}(s),\qquad\tau>0.
		\]
		Applying \eqref{E:zeta.phi} to a graded automorphism $\phi\in\Aut_\grr(\goe)$ that acts as an isometric reflection on $\goe_{-1}$, we obtain $\zeta_{\rho_{-\hbar}}(s)=\zeta_{\rho_\hbar}(s)$.
		Hence,
		\[
			\zeta_{\rho_\hbar}(s)
			=|\hbar|^{-\kappa s}\zeta_{\rho_1}(s)
		\]
		for all $0\neq\hbar\in\R$.
		This immediately yields \eqref{E:factor.II}, cf.~\eqref{E:zeta.II}.
		In view of \cite[Theorem~6, Eqs.~(200) and (202)]{H23} the function $\zeta_{\rho_1}(s)$ is meromorphic on the entire complex plane and its poles can only be located at $s=\frac{1-2j}\kappa$, $j\in\N_0$.
		
		The Epstein \cite[p.~620]{E03} zeta function $Z_\Epst(s;a)$ converges for $\Re s>1$, it extends to a meromorphic function on the entire complex plane which has a single (simple) pole at $s=1$ of residue $2$, it vanishes at $s\in-2\N$, and it satisfies
		\begin{equation}\label{E:qwerty999}
			Z_\Epst(0;a)=
			\begin{cases}
				-1&\text{if $a\in\Z$, and}\\
				0&\text{otherwise.}
			\end{cases}
		\end{equation}
		Combining this with \eqref{E:factor.II} we see that the sum in \eqref{E:zeta.II} converges for $\Re s>2/\kappa$ and that $\zeta_{\II,\Gamma,\chi}(s)$ extends to a meromorphic function on the entire complex plane which has a simple pole at $s=2/\kappa$ with residue $\frac1r\cdot\frac2\kappa\cdot\zeta_{\rho_1}(2/\kappa)$.
		If $c\notin r\Z$ then the zeros of $Z_\Epst(\kappa s-1;\frac cr)$ cancel all the poles of $\zeta_{\rho_1}(s)$ and $\zeta_{\II,\Gamma,\chi}(s)$ is holomorphic for $s\neq2/\kappa$.
		If $c\in r\Z$ then all but one pole get canceled and $\zeta_{\II,\Gamma,\chi}(s)$ has one further pole at $s=1/\kappa$ with residue $-\res_{s=1/\kappa}\zeta_{\rho_1}(s)$.

		From \cite[Proposition~1, Theorem~3]{H23} we obtain
		\begin{equation}\label{E:qwerty0}
			\zeta_{\rho_\hbar}(0)=0
			\qquad\text{and}\qquad
			\zeta_{\rho_\hbar}'(0)=0.
		\end{equation}
		Combining this with \eqref{E:factor.II} and the fact that $\zeta_\Epst(s;a)$ is holomorphic at $s=-1$ we obtain \eqref{E:zeta.II.0}.
	\end{proof}

	\begin{remark}\label{R:Epst.Hurw}
		Recall that
		\[
			Z_\Epst(s;a)
			=\begin{cases}
				\zeta_\Hurw(s,a)+\zeta_\Hurw(s,1-a)&\text{if $0<a<1$, and}\\
				2\zeta_\Riem(s)&\text{if $a=0$.}
			\end{cases}
		\]
		Here $\zeta_\Hurw(s,a)=\sum_{n=0}^\infty(n+a)^{-s}$ denotes the Hurwitz zeta function, $a>0$, and $\zeta_\Riem(s)=\zeta_\Hurw(s,1)=\sum_{n=1}^\infty n^{-s}$ denotes the Riemann zeta function.
		In particular, for $0\leq a<1$,
		\[
			\zeta_\Epst(-1;a)=-B_2(a),
		\]
		where $B_2(a)=a^2-a+\frac16$ denotes the second Bernoulli polynomial.
		Indeed, this follows from the classical identities $\zeta_\Hurw(-1,a)=-\tfrac12B_2(a)$, $B_2(1-a)=B_2(a)$, and $\zeta_\Riem(-1)=-\frac1{12}=-\tfrac12B_2(0)$.
	\end{remark}

	\begin{lemma}\label{L:zeta.semi}
		Suppose $(0,0)\neq(\lambda,\mu)\in\R^2$, $\nu_0\in\R$, and $0\neq d\in\R$.
		Then 
		\[
			\sum_{n\in\Z}\zeta_{\rho_{\lambda,\mu,\nu_0+2dn}}(s)
		\] 
		converges for $\Re s>2/\kappa$, and this function admits a meromorphic continuation to the half plane\footnote{Below we will see that this function in fact extends to a meromorphic function on the entire complex plane whose poles can only be located at $s=(2-3j)/\kappa$, $j\in\N_0$, cf.~\eqref{E:zeta.k}.} $\Re s>-1/8\kappa$ which is holomorphic at $s=0$.
		Moreover,
		\[
			\Bigl(\sum_{n\in\Z}\zeta_{\rho_{\lambda,\mu,\nu_0+2dn}}\Bigr)(0)=0
			\qquad\text{and}\qquad
			\Bigl(\sum_{n\in\Z}\zeta_{\rho_{\lambda,\mu,\nu_0+2dn}}\Bigr)'(0)=0.
		\]
	\end{lemma}

	\begin{proof}
		The homogeneity in \eqref{E:zeta.phit} gives
		\begin{equation}\label{E:zeta.homog}
			\zeta_{\rho_{\tau^3\lambda,\tau^3\mu,\tau^4\nu}}(s)
			=\tau^{-2\kappa s}\zeta_{\rho_{\lambda,\mu,\nu}}(s),\qquad\tau>0.
		\end{equation}
		W.l.o.g.\ we may thus assume $\lambda^2+\mu^2=1$.
		Furthermore, we may assume $d>0$ and $-d\leq\nu_0\leq d$.
		From \cite[Theorems~4 and 5(III)]{H23} we know
		\begin{equation}\label{E:zetap.lmn.0}
			\zeta_{\rho_{\lambda,\mu,\nu}}(0)=0
			\qquad\text{and}\qquad
			\zeta_{\rho_{\lambda,\mu,\nu}}'(0)=0
		\end{equation}
		for all $\nu\in\R$.

		For $\nu\leq-d$ we have
		\begin{equation}\label{E:zeta-}
			\zeta_{\rho_{\lambda,\mu,\nu}}(s)=|\nu|^{-\kappa s/2}E(s)+|\nu|^{-2\kappa s+3/2}C_-(s)+R_{\nu,-}(s)
		\end{equation}
		where $E(s)$, $C_-(s)$ and $R_{\nu,-}(s)$ are meromorphic functions on the entire complex plane 
		whose poles are all contained in the set
		\[
			P:=\left(\left\{\frac{3-j}{4\kappa}:j\in\N_0\right\}\cup\left\{\frac{l/2+1-k}\kappa:k,l\in\N_0\right\}\right)\setminus(-\N_0),
		\]
		and the estimate
		\begin{equation}\label{E:R-.esti}
			R_{\nu,-}(s)=O\left(|\nu|^{-5/4}\right)
		\end{equation}
		holds uniformly on compact subsets of $\{s\in\mathbb C\setminus P:\Re s\geq-1/8\kappa\}$ and for $\nu\leq-d$.
		This follows from the first estimate in \cite[Corollary~1]{H23} by taking, with the notation there, $N=7/4$, $\sigma=-1/8\kappa$, and $\varepsilon=d$.
		Combining \eqref{E:zetap.lmn.0}, \eqref{E:zeta-}, and \eqref{E:R-.esti}, we see that
		\begin{equation}\label{E:ECR-0}
			E(0)=C_-(0)=R_{\nu,-}(0)=0
			\qquad\text{and}\qquad
			E'(0)=C'_-(0)=R'_{\nu,-}(0)=0.
		\end{equation}
		From \eqref{E:zeta-} we obtain
		\begin{multline}\label{E:zeta.cont-}
			\sum_{n=1}^\infty\zeta_{\rho_{\lambda,\mu,\nu_0-2dn}}(s)
			=(2d)^{-\kappa s/2}\zeta_\Hurw\bigl(\tfrac{\kappa s}2;1-\tfrac{\nu_0}{2d}\bigr)E(s)
			\\
			+(2d)^{-2\kappa s+3/2}\zeta_\Hurw\bigl(2\kappa s-\tfrac32;1-\tfrac{\nu_0}{2d}\bigr)C_-(s)
			+\sum_{n=1}^\infty R_{\nu_0-2dn,-}(s)
		\end{multline}
		where the last sum on the right hand side converges uniformly on compact subsets of $\{s\in\mathbb C\setminus P:\Re s\geq-1/8\kappa\}$ by the estimate in \eqref{E:R-.esti}.
		The sums making up the Hurwitz zeta functions converge for $\Re s>2/\kappa$ and $\Re s>5/4\kappa$, respectively.
		Hence, $\sum_{n=1}^\infty\zeta_{\rho_{\lambda,\mu,\nu_0-2dn}}(s)$ converges for $\Re s>2/\kappa$, and the right hand side in \eqref{E:zeta.cont-} provides the meromorphic continuation to $\Re s>-1/8\kappa$.
		Combining this with \eqref{E:ECR-0} and the fact that the Hurwitz zeta function $\zeta_\Hurw(s;a)$ is holomorphic at $s=0$ and at $s=-3/2$, we conclude
		\begin{equation}\label{E:zeta-sum}
			\Bigl(\sum_{n=1}^\infty\zeta_{\rho_{\lambda,\mu,\nu_0-2dn}}\Bigr)(0)=0
			\qquad\text{and}\qquad
			\Bigl(\sum_{n=1}^\infty\zeta_{\rho_{\lambda,\mu,\nu_0-2dn}}\Bigr)'(0)=0.
		\end{equation}

		For $\nu\geq d$ we have
		\begin{equation}\label{E:zeta+}
			\zeta_{\rho_{\lambda,\mu,\nu}}(s)
			=\nu^{-2\kappa s+3/2}C_+(s)+R_{\nu,+}(s)
		\end{equation}
		where $C_+(s)$ and $R_{\nu,+}(s)$ are meromorphic functions on the entire complex plane 
		whose poles are all contained in the set $P$, and the estimate
		\begin{equation}\label{E:R+.esti}
			R_{\nu,+}(s)=O\left(\nu^{-5/4}\right)
		\end{equation}
		holds uniformly on compact subsets of $\{s\in\mathbb C\setminus P:\Re s\geq-1/8\kappa\}$ and for $\nu\geq d$.
		This follows from the second estimate in \cite[Corollary~1]{H23} by taking again, with the notation there, $N=7/4$, $\sigma=-1/8\kappa$, and $\varepsilon=d$.
		Combining \eqref{E:zetap.lmn.0}, \eqref{E:zeta+}, and \eqref{E:R+.esti}, we see that
		\begin{equation}\label{E:ECR+0}
			C_+(0)=R_{\nu,+}(0)=0
			\qquad\text{and}\qquad
			C'_+(0)=R'_{\nu,+}(0)=0.
		\end{equation}
		From \eqref{E:zeta+} we obtain
		\begin{multline}\label{E:zeta.cont+}
			\sum_{n=1}^\infty\zeta_{\rho_{\lambda,\mu,\nu_0+2dn}}(s)
			=(2d)^{-2\kappa s+3/2}\zeta_\Hurw\bigl(2\kappa s-\tfrac32;1+\tfrac{\nu_0}{2d}\bigr)C_+(s)
			\\
			+\sum_{n=1}^\infty R_{\nu_0+2dn,+}(s)
		\end{multline}
		where the last sum on the right hand side converges on compact subsets of $\{s\in\mathbb C\setminus P:\Re s\geq-1/8\kappa\}$ by the estimate in \eqref{E:R+.esti}.
		As before we conclude that $\sum_{n=1}^\infty\zeta_{\rho_{\lambda,\mu,\nu_0+2dn}}(s)$ converges for $\Re s>2/\kappa$, and the right hand side in \eqref{E:zeta.cont+} provides the meromorphic continuation to $\Re s>-1/8\kappa$.
		Using \eqref{E:ECR+0} this yields
		\begin{equation*}
			\Bigl(\sum_{n=1}^\infty\zeta_{\rho_{\lambda,\mu,\nu_0+2dn}}\Bigr)(0)=0
			\qquad\text{and}\qquad
			\Bigl(\sum_{n=1}^\infty\zeta_{\rho_{\lambda,\mu,\nu_0+2dn}}\Bigr)'(0)=0.
		\end{equation*}
		Combining this with \eqref{E:zeta-sum} and \eqref{E:zetap.lmn.0} we obtain the lemma.
	\end{proof}

	\begin{lemma}\label{L:zeta.III}
		The sum in \eqref{E:zeta.III} converges for $\Re s>10/2\kappa$, and 
		\begin{equation}\label{E:factor.III}
			\zeta_{\III,\Gamma,\chi}(s)
			=\tfrac1r\cdot Z^{(\Gamma'')^*}_\Epst\bigl(\tfrac{2\kappa s-4}3;\lambda_0,\mu_0\bigr)\cdot f(s)+\hat R(s),
		\end{equation}
		where
		\begin{equation}\label{E:Z.def}
			Z_\Epst^{(\Gamma'')^*}\bigl(s;\lambda_0,\mu_0\bigr)
			:=\sideset{}{'}\sum_{\substack{(\lambda,\mu)\in\\(\lambda_0,\mu_0)+(\Gamma'')^*}}(\lambda^2+\mu^2)^{-s/2}
		\end{equation}
		denotes an Epstein zeta function, $f(s)$ denotes a meromorphic function defined in \eqref{E:f} below, and $\hat R(s)$ is an entire function.
		In particular, $\zeta_{\III,\Gamma,\chi}(s)$ extends to a meromorphic function on the entire complex plane which has at most a simple pole at $s=10/2\kappa$ with residue $\frac{3\pi}{\kappa r}\cdot\Area(\R^2/\Gamma'')\cdot f(10/2\kappa)$.
		\footnote{If $f(10/2\kappa)$ vanishes then $\zeta_{\III,\Gamma,\chi}(s)$ is regular at $s=10/2\kappa$.}
		If $\chi|_{\Gamma\cap Z}\neq1$, then this is the only pole of $\zeta_{\III,\Gamma,\chi}(s)$.
		If $\chi|_{\Gamma\cap Z}=1$, then $\zeta_{\III,\Gamma,\chi}(s)$ has one further (simple) pole at $s=2/\kappa$ with residue $-\frac1r\cdot\res_{s=2/\kappa}f(s)$.
		\footnote{We do not rule out that this residue vanishes either.}
		Moreover,
		\begin{equation}\label{E:zeta.III.0}
			\zeta_{\III,\Gamma,\chi}(0)=0
			\qquad\text{and}\qquad
			\zeta_{\III,\Gamma,\chi}'(0)=0.
		\end{equation}
	\end{lemma}

	\begin{proof}
		There exists a graded automorphism $\phi\in\Aut_\grr(\goe)$ preserving $g$ and acting as a rotation on $\goe_{-3}$ such that $\rho_{\lambda,\mu,\nu}\circ\phi$ is unitarily equivalent to $\rho_{\sqrt{\lambda^2+\mu^2},0,\nu}$, cf.~\cite[Eq.~(48)]{H23}.
		Hence, \cite[Section~3.1]{H23}
		\begin{equation}\label{E:rot.III}
			\tr\left(e^{-t\rho_{\lambda,\mu,\nu}(\Delta_q)}\right)
			=\tr\left(e^{-t\rho_{\sqrt{\lambda^2+\mu^2},0,\nu}(\Delta_q)}\right).
		\end{equation}
		For $t>0$ put
		\begin{equation}\label{E:theta.def}
			\vartheta_{\lambda,\mu,\nu}(t)
			:=\str\left(Ne^{-t\rho_{\lambda,\mu,\nu}(\Delta)}\right)
			:=\sum_{q=0}^5(-1)^qN_q\tr\left(e^{-t\rho_{\lambda,\mu,\nu}(\Delta_q)}\right).
		\end{equation}
		Using \eqref{E:rot.III} and \cite[Equation~(203)]{H23} we conclude that there exist a Schwartz function $k\in\mathcal S(\R^3)$ such that 
		\begin{equation}\label{E:theta.k}
			\vartheta_{\lambda,\mu,\nu}(t)
			=\frac{t^{-3/4\kappa}}{\sqrt[4]{\lambda^2+\mu^2}}\int_{-\infty}^\infty k\begin{pmatrix}\frac{t^{2/4\kappa}\nu}{2\sqrt{\lambda^2+\mu^2}}+\frac{x^2}2\\t^{3/4\kappa}\sqrt[4]{\lambda^2+\mu^2}x\\t^{6/4\kappa}\sqrt{\lambda^2+\mu^2}\end{pmatrix}dx.
		\end{equation}
		According to \cite[Eq.~(188)]{H23} the function $k$ enjoys the symmetry
		\begin{equation}\label{E:k.sym}
			k\begin{pmatrix}-x_1\\x_2\\-x_3\end{pmatrix}
			=k\begin{pmatrix}x_1\\x_2\\x_3\end{pmatrix}.
		\end{equation}
		Note that $\vartheta_{\lambda,\mu,\nu}(t)$ is a Schwartz function in $\nu$, for fixed $(\lambda,\mu)\neq(0,0)$ and $t>0$.

		By the Euler--Maclaurin summation formula, for $\nu_0\in\R$ and $0\neq d\in\R$,
		\begin{multline*}
			\sum_{n\in\Z}\vartheta_{\lambda,\mu,\nu_0+2dn}(t)
			=\int_\R\vartheta_{\lambda,\mu,\nu_0+2d\nu}(t)d\nu
			\\+\frac{(-1)^{N+1}}{N!}\int_\R\left(\tfrac{\partial^N}{\partial\nu^N}\vartheta_{\lambda,\mu,\nu_0+2d\nu}(t)\right)P_N(\nu)d\nu
		\end{multline*}
		where $P_N(x)=B_N(x-\lfloor x\rfloor)$ denotes the periodic extension of the $N$-th Bernoulli polynomial restricted to the unit interval.
		Substituting $\nu_0+2d\nu\leftrightarrow\nu$ in the integrals, this yields
		\begin{multline*}
			\sum_{n\in\Z}\vartheta_{\lambda,\mu,\nu_0+2dn}(t)
			=\frac1{2d}\int_\R\vartheta_{\lambda,\mu,\nu}(t)d\nu
			\\+\frac{(-1)^{N+1}(2d)^{N-1}}{N!}\int_\R\left(\tfrac{\partial^N}{\partial\nu^N}\vartheta_{\lambda,\mu,\nu}(t)\right)P_N\left(\frac{\nu-\nu_0}{2d}\right)d\nu.
		\end{multline*}
		Plugging in \eqref{E:theta.k}, we obtain
		\begin{multline*}
			\sum_{n\in\Z}\vartheta_{\lambda,\mu,\nu_0+2dn}(t)
			=\frac{t^{-3/4\kappa}}{2d\sqrt[4]{\lambda^2+\mu^2}}\int_{\R^2}k\begin{pmatrix}\frac{t^{2/4\kappa}\nu}{2\sqrt{\lambda^2+\mu^2}}+\frac{x^2}2\\t^{3/4\kappa}\sqrt[4]{\lambda^2+\mu^2}x\\t^{6/4\kappa}\sqrt{\lambda^2+\mu^2}\end{pmatrix}dxd\nu
			\\+\frac{(-1)^{N+1}d^{N-1}t^{-3/4\kappa}t^{N/2\kappa}}{2N!\sqrt[4]{\lambda^2+\mu^2}(\lambda^2+\mu^2)^{N/2}}
			\int_{\R^2}\frac{\partial^Nk}{\partial x_1^N}\begin{pmatrix}\frac{t^{2/4\kappa}\nu}{2\sqrt{\lambda^2+\mu^2}}+\frac{x^2}2\\t^{3/4\kappa}\sqrt[4]{\lambda^2+\mu^2}x\\t^{6/4\kappa}\sqrt{\lambda^2+\mu^2}\end{pmatrix}P_N\left(\frac{\nu-\nu_0}{2d}\right)dxd\nu.
		\end{multline*}
		Substituting $\Bigl(\frac{t^{2/4\kappa}\nu}{2\sqrt{\lambda^2+\mu^2}}+\frac{x^2}2,t^{3/4\kappa}\sqrt[4]{\lambda^2+\mu^2}x\Bigr)\leftrightarrow(\nu,x)$ in the integrals, this yields
		\begin{multline}\label{E:qwerty}
			\sum_{n\in\Z}\vartheta_{\lambda,\mu,\nu_0+2dn}(t)
			=\frac{t^{-4/2\kappa}}d\int_{\R^2}k\begin{pmatrix}\nu\\ x\\t^{3/2\kappa}\sqrt{\lambda^2+\mu^2}\end{pmatrix}dxd\nu
			\\+\frac{(-1)^{N+1}d^{N-1}t^{(N-4)/2\kappa}}{N!(\lambda^2+\mu^2)^{N/2}}
			\int_{\R^2}\frac{\partial^Nk}{\partial x_1^N}\begin{pmatrix}\nu\\x\\t^{3/2\kappa}\sqrt{\lambda^2+\mu^2}\end{pmatrix}
			\\\cdot P_N\left(\frac{2t^{-1/2\kappa}\sqrt{\lambda^2+\mu^2}\nu-t^{-2/\kappa}x^2-\nu_0}{2d}\right)dxd\nu.
		\end{multline}
		From \eqref{E:zeta.rho.def} and \eqref{E:theta.def} we have
		\[
			\zeta_{\rho_{\lambda,\mu,\nu}}(s)
			=\frac1{\Gamma(s)}\int_0^\infty t^{s-1}\vartheta_{\lambda,\mu,\nu}(t)dt.
		\]
		Plugging in \eqref{E:qwerty} and substituting $t^{3/2\kappa}\sqrt{\lambda^2+\mu^2}\leftrightarrow t^{3/2\kappa}$ in the integrals, we get
		\begin{equation}\label{E:zeta.k}
			\sum_{n\in\Z}\zeta_{\rho_{\lambda,\mu,\nu_0+2dn}}(s)
			=\frac{(\lambda^2+\mu^2)^{(2-\kappa s)/3}}d\cdot f(s)+R_{\lambda,\mu,\nu_0,d,N}(s)
		\end{equation}
		where 
		\begin{equation}\label{E:f}
			f(s):=\frac1{\Gamma(s)}\int_0^\infty t^{s-1}t^{-4/2\kappa}\int_{\R^2}k\begin{pmatrix}\nu\\ x\\t^{3/2\kappa}\end{pmatrix}dxd\nu dt
		\end{equation}
		and
		\begin{multline}\label{E:R}
			R_{\lambda,\mu,\nu_0,d,N}(s)
			:=\frac{(-1)^{N+1}d^{N-1}(\lambda^2+\mu^2)^{(2-2N-\kappa s)/3}}{N!\Gamma(s)}
			\int_0^\infty t^{s-1}t^{(N-4)/2\kappa}
			\\\cdot\int_{\R^2}\frac{\partial^Nk}{\partial x_1^N}\begin{pmatrix}\nu\\x\\t^{3/2\kappa}\end{pmatrix}
				P_N\left(\frac{(\lambda^2+\mu^2)^{2/3}(2t^{-1/2\kappa}\nu-t^{-2/\kappa}x^2)-\nu_0}{2d}\right)dxd\nu dt.
		\end{multline}
		The integral in \eqref{E:f} converges for $\Re s>4/2\kappa$ and extends to a meromorphic function on the entire complex plane which has only simple poles and these can only be located at $s=(2-3j)/\kappa$, $j\in\N_0$.
		Note here that the inner integral over $\R^2$ in \eqref{E:f} results in a Schwartz function in the variable $t^{3/2\kappa}$ which is even according to the symmetry in \eqref{E:k.sym}.
		The integral in \eqref{E:R} converges for $\Re s>(4-N)/2\kappa$ and the estimate
		\begin{equation}\label{E:R.esti}
			R_{\lambda,\mu,\nu_0,d,N}(s)=O\left(d^{N-1}(\lambda^2+\mu^2)^{(2-2N-\kappa\Re s)/3}\right),
		\end{equation}
		holds uniformly for $(\lambda,\mu)\neq(0,0)$, $\nu_0$, $d\neq0$ and uniformly for $s$ in compact subsets contained in $\Re s>(4-N)/2\kappa$.
		Hence, $\sum_{n\in\Z}\zeta_{\rho_{\lambda,\mu,\nu_0+2dn}}(s)$ converges for $\Re s>4/2\kappa$ and \eqref{E:zeta.k} provides the meromorphic extension to the entire complex plane whose poles can only occur at $s=(2-3j)/\kappa$, $j\in\N_0$.
		Assuming $N>4$ and using Lemma~\ref{L:zeta.semi} we obtain from \eqref{E:zeta.k} and \eqref{E:R.esti}
		\begin{equation}\label{E:fR.0}
			f(0)=R_{\lambda,\mu,\nu_0,d,N}(0)=0
			\qquad\text{and}\qquad
			f'(0)=R_{\lambda,\mu,\nu_0,d,N}'(0)=0.
		\end{equation}

		Now suppose $(0,0)\neq(\lambda,\mu)\in(\lambda_0,\mu_0)+(\Gamma'')^*$ as in \eqref{E:zeta.III}.
		In particular, $\lambda$ and $\mu$ are integers divisible by $r$, and so is $d=\gcd(\lambda,\mu)$.
		Also recall the integer $w$ defined in \eqref{E:w.def}, and the real number $\nu_0$ is given by \eqref{E:nu0.def}.
		Using \eqref{E:M.mult} and \eqref{E:m.total} we get from \eqref{E:zeta.k}	
		\begin{equation}\label{E:zeta.k2}
			\sum_{n\in\Z}m\bigl(\tfrac dr,r,w,n\bigr)\zeta_{\rho_{\lambda,\mu,\nu_0+rn}}(s)
			=\frac{(\lambda^2+\mu^2)^{(2-\kappa s)/3}}r\cdot f(s)
			+\tilde R_{\lambda,\mu,N}(s)
		\end{equation}
		where the sum on the left hand side converges for $\Re s>4/2\kappa$ and 
		\[
			\tilde R_{\lambda,\mu,N}(s)
			:=\sum_{n=1}^{2d/r}m\bigl(\tfrac dr,r,w,n\bigr)R_{\lambda,\mu,\nu_0+rn,d,N}(s).
		\]
		Assuming $N>4$ we obtain from \eqref{E:fR.0}
		\begin{equation}\label{E:tR.0}
			\tilde R_{\lambda,\mu,N}(0)=0
			\qquad\text{and}\qquad
			\tilde R_{\lambda,\mu,N}'(0)=0.
		\end{equation}
		From \eqref{E:R.esti} and \eqref{E:m.total}, using the obvious estimate $|d|\leq\sqrt{\lambda^2+\mu^2}$, we get
		\begin{equation}\label{E:tR.esti}
			\tilde R_{\lambda,\mu,N}(s)=O\left((\lambda^2+\mu^2)^{(4-N-2\kappa\Re s)/6}\right),
		\end{equation}
		uniformly for $(0,0)\neq(\lambda,\mu)\in(\lambda_0,\mu_0)+(\Gamma'')^*$, and uniformly for $s$ in compact subsets contained in $\Re s>(4-N)/2\kappa$.
		Summing over all $(\lambda,\mu)\neq(0,0)$ in the shifted lattice, we obtain from \eqref{E:zeta.III} and \eqref{E:zeta.k2},
		\begin{equation}\label{E:zeta.III.fR2}
			\zeta_{\III,\Gamma,\chi}(s)
			=\tfrac1r\cdot Z^{(\Gamma'')^*}_\Epst\bigl(\tfrac{2\kappa s-4}3;\lambda_0,\mu_0\bigr)\cdot f(s)+\hat R_N(s),
		\end{equation}
		where the Epstein zeta function is defined in \eqref{E:Z.def}, and
		\begin{equation}\label{E:hatR.def}
			\hat R_N
			:=\sideset{}{'}\sum_{\substack{(\lambda,\mu)\in\\(\lambda_0,\mu_0)+(\Gamma'')^*}}\tilde R_{\lambda,\mu,N}(s).
		\end{equation}
		The sum in \eqref{E:hatR.def} converges uniformly on compact subsets of $\Re s>(10-N)/2\kappa$ by the estimate in \eqref{E:tR.esti}.
		Assuming $N>10$ we obtain from \eqref{E:tR.0} and \eqref{E:hatR.def},
		\begin{equation}\label{E:hatR.0}
			\hat R_N(0)=0
			\qquad\text{and}\qquad
			\hat R_N'(0)=0.
		\end{equation}
		As $\hat R_N(s)$ is an entire function which is independent of $N$, we obtain \eqref{E:factor.III}.

		Using a basis of the lattice $(\Gamma'')^*$, we may write
		\[
			Z_\Epst^{(\Gamma'')^*}\bigl(s;\lambda_0,\mu_0\bigr)
			=\sideset{}{'}\sum_{(\alpha,\beta)\in(\alpha_0,\beta_0)+\Z^2}\bigl(\varphi(\alpha,\beta)\bigr)^{-s/2}
		\]
		where $(\alpha_0,\beta_0)\in\R^2$ corresponds to $(\lambda_0,\mu_0)$, and $\varphi(\alpha,\beta)$ is a positive quadratic form with $\sqrt{\det\varphi}=\Area((\R^2)^*/(\Gamma'')^*)=1/\Area(\R^2/\Gamma'')$.
		Hence, this Epstein \cite[p.~627]{E03} zeta function converges for $\Re s>2$, it extends to a meromorphic function on the entire complex plane which has a single (simple) pole at $s=2$ with residue $2\pi/\sqrt{\det\varphi}=2\pi\cdot\Area(\R^2/\Gamma'')$, it vanishes at $s\in-2\N$, and it satisfies
		\begin{equation}\label{E:qwerty9999}
			Z^{(\Gamma'')^*}_\Epst(0;\lambda_0,\mu_0)=
			\begin{cases}
				-1&\text{if $(\lambda_0,\mu_0)\in\Gamma''$, and}\\
				0&\text{otherwise.}
			\end{cases}
		\end{equation}
		Note that $(\lambda_0,\mu_0)\in\Gamma''$ if and only if $\chi|_{\Gamma\cap Z}=1$ by Lemma~\ref{L:lm0} and Lemma~\ref{L:Richardson:235}(II).
		Combining this with \eqref{E:factor.III} we see that the sum in \eqref{E:zeta.III} converges for $\Re s>10/2\kappa$ and that $\zeta_{\III,\Gamma,\chi}(s)$ extends to a meromorphic function on the entire complex plane which has a simple pole at $s=10/2\kappa$ with residue $\frac1r\cdot\frac3{2\kappa}\cdot2\pi\cdot\Area(\R^2/\Gamma'')\cdot f(10/2\kappa)$.
		If $\chi|_{\Gamma\cap Z}\neq1$, then the zeros of $Z_\Epst^{(\Gamma'')^*}(\frac{2\kappa s-4}3;\lambda_0,\mu_0)$ cancel all the poles of $f(s)$ and $\zeta_{\III,\Gamma,\chi}(s)$ is holomorphic for $s\neq10/2\kappa$.
		If $\chi|_{\Gamma\cap Z}=1$, then all but one pole get canceled and $\zeta_{\III,\Gamma,\chi}(s)$ has one further pole at $s=2/\kappa$ with residue $-\frac1r\cdot\res_{s=2/\kappa}f(s)$.
		As $Z_\Epst^{(\Gamma'')^*}\bigl(s;\lambda_0,\mu_0\bigr)$ is holomorphic at $s=-4/3$ we obtain \eqref{E:zeta.III.0} by combining \eqref{E:zeta.III.fR2} with \eqref{E:fR.0} and \eqref{E:hatR.0}.
	\end{proof}

	\begin{remark}
		It is known that $\zeta_{\Gamma\setminus G,\mathcal D_{\Gamma\setminus G},F_\chi,g_{\Gamma\setminus G},h_\chi}(s)$ has a single pole at $s=10/2\kappa$, see \cite[proof of Lemma~4.2]{H22} or \cite[Theorem~1.1]{F22}.
		Hence, in view of \eqref{E:zeta.deco} the poles at $s=1/\kappa$ and $s=2/\kappa$ of $\zeta_{\I,\Gamma,\chi}(s)$ and $\zeta_{\II,\Gamma,\chi}(s)$ must cancel.
		By Lemmas~\ref{L:zeta.I}, \ref{L:zeta.II}, and \ref{L:zeta.III} this is the case if and only if
		\begin{equation}\label{E:cons1}
			\res_{s=1/\kappa}\zeta_{\rho_1}(s)=\frac\pi\kappa\cdot\zeta_{\rho_{1,0}}(1/\kappa)
		\end{equation}
		and
		\begin{equation}\label{E:cons2}
			\res_{s=2/\kappa}f(s)=\frac2\kappa\cdot\zeta_{\rho_1}(2/\kappa).
		\end{equation}
		Moreover, we must have
		\begin{equation}\label{E:cons3}
			\res_{s=10/2\kappa}\zeta_{\Gamma\setminus G,\mathcal D_{\Gamma\setminus G},F_\chi,g_{\Gamma\setminus G},h_\chi}(s)
			=\frac{3\pi}{\kappa r}\cdot\Area(\R^2/\Gamma'')\cdot f(10/2\kappa).
		\end{equation}

		Proceeding as in \cite[Section~4]{H23} one readily obtains the spectra of $\rho_{1,0}(\Delta_q)$.
		For $q=0,5$ we only have the eigenvalue $(2\pi)^{2\kappa}$ with multiplicity one; 
		for $q=1,4$ we only have the eigenvalue $(2\pi)^{2\kappa}$ with multiplicity two;
		and for $q=2,3$ we have the eigenvalue $(2\pi)^{2\kappa}$ with multiplicity one and the eigenvalue $(2\pi)^{2\kappa}2^{-\kappa/2}$ with multiplicity two.
		Plugging this into \eqref{E:zeta.rho.def} we obtain
		\begin{equation}\label{E:cons4}
			\zeta_{\rho_{1,0}}(s)
			=(2\pi)^{-2\kappa s}\cdot 4\left(1-2^{\kappa s/2}\right).
		\end{equation}

		In \cite[Section~5]{H23} the spectra of $\rho_\hbar(D_q^*D_q)$ have been determined explicitly.
		Using the computations in \cite[Section~5.4]{H23} we see that $\tr\rho_\hbar(\Delta_q)^{-s}$ has a simple pole at $s=1/\kappa$.
		For $q=0,5$ the residue is $1/4\pi\hbar\kappa$, for $q=1,4$ the residue is $1/2\pi\hbar\kappa$, and for $q=2,3$ the residue is $1/4\pi\hbar\kappa+\sqrt2/2\pi\kappa$.
		Using \eqref{E:zeta.rho.def} this yields 
		\begin{equation}\label{E:cons5}
			\res_{s=1/\kappa}\zeta_{\rho_\hbar}(s)=\frac{1-\sqrt2}{\pi\kappa\hbar}
		\end{equation}

		Combining \eqref{E:cons4} with \eqref{E:cons5} we do indeed get \eqref{E:cons1}.
		We will not attempt to verify \eqref{E:cons2} and \eqref{E:cons3} independently here.
	\end{remark}

\section{Proof of the main theorem}\label{S:proof}

	In order to prove the theorem formulated in the introduction, note first that we may w.l.o.g.\ assume $\Gamma$, $\mathcal D_{\Gamma\setminus G}$, and $g_{\Gamma\setminus G}$ to be of standard form.
	Indeed, by Lemma~\ref{L:all.lat}, we may assume that the lattice $\Gamma$ is of the form considered in Section~\ref{S:lat}, i.e., generated by the exponentials of $\tilde\gamma_i$ as in \eqref{E:Gamma.tgen}.
	Moreover, according to \cite[Lemma~4.2]{H22} the torsion of the Rumin complex does not depend on the choice of $\mathcal D_{\Gamma\setminus G}$ and $g_{\Gamma\setminus G}$, as long as they are induced from a left invariant (2,3,5) distribution on $G$ and a left invariant fiberwise graded Euclidean inner product on $\mathfrak tG$, respectively.
	Hence we may also assume that $\mathcal D_{\Gamma\setminus G}$ and $g_{\Gamma\setminus G}$ are of the form considered in Section~\ref{S:zeta.deco}, i.e., $\mathcal D_{\Gamma\setminus G}$ is induced by the left invariant 2-plane field $\mathcal D_G$ on $G$ spanned by $X_1$ and $X_2$, and $g_{\Gamma\setminus G}$ is induced from a graded Euclidean inner product $g$ on $\goe$ such that $X_1,\dotsc,X_5$ are orthonormal.
	Strictly speaking, \cite[Lemma~4.2]{H22} only covers graded Euclidean inner products induced from left invariant sub-Riemannian metrics on $\mathcal D_G$, but the proof there can readily be extended to the generality required here.

	As the Rumin complex computes the de~Rham cohomology $H^*(\Gamma\setminus G;F_\chi)$, one can use the Leray--Serre spectral sequence to prove the acyclicity of the Rumin complex, cf.~Lemma~\ref{L:RS.Tor} below.
	However, the acyclicity can also be read off the decomposition provided in Lemma~\ref{L:Richardson:235}.
	Indeed, the irreducible unitary representations appearing in this decomposition are all nontrivial, as $\chi$ is assumed to be nontrivial.
	Since the Rumin complex is a Rockland complex, it becomes exact in every nontrivial unitary representation of $G$.
	The acyclicity of the Rumin complex thus follows from the decomposition in \eqref{E:comp.deco}.

	Combining Lemma~\ref{L:zeta.deco} with Equations~\eqref{E:zeta.I.0}, \eqref{E:zeta.II.0}, and \eqref{E:zeta.III.0} yields
	\[
		\zeta_{\Gamma\setminus G,\mathcal D_{\Gamma\setminus G},F_\chi,g_{\Gamma\setminus G},h_\chi}'(0)
		=\zeta_{\I,\Gamma,\chi}'(0)+\zeta_{\II,\Gamma,\chi}'(0)+\zeta_{\III,\Gamma,\chi}'(0)=0.
	\]
	Using \eqref{E:tau.via.zeta} we obtain
	\[
		\tau(\Gamma\setminus G,\mathcal D_{\Gamma\setminus G},F_\chi,g_{\Gamma\setminus G},h_\chi)=1.
	\] 
	This completes the proof of the theorem.

	To derive the corollary we need to know the Ray--Singer torsion of $\Gamma\setminus G$.

	\begin{lemma}\label{L:RS.Tor}
		If $\chi\colon\Gamma\to U(1)$ is a nontrivial character, then $H^*(\Gamma\setminus G;F_\chi)=0$ and the Ray--Singer torsion is trivial, $\tau_\RS(\Gamma\setminus G;F_\chi)=1$.
	\end{lemma}

	\begin{proof}
		Suppose for now that $\chi$ does not vanish on $\Gamma\cap[G,G]$.
		Recall from \cite[Section~2]{R72} or Lemma~\ref{L:Gamma} that $\Gamma\cap[G,G]$ is a lattice in $[G,G]$ and so is the image of $\Gamma$ under the canonical homomorphism $G\to G/[G,G]\cong\R^2$ with (abelian) fiber $[G,G]$.
		Hence, modding out the lattice, we obtain a fibration $p\colon\Gamma\setminus G\to B$ with a 2-torus $B\cong T^2$ as a base and with typical fiber a 3-torus, $V:=\frac{[G,G]}{\Gamma\cap[G,G]}\cong T^3$.
		Using the K\"unneth theorem one readily shows that the cohomology of a torus $T^n\cong\R^n/\Z^n$ with coefficients in any nontrivial flat complex line bundle is acyclic.
		As $\chi$ is nontrivial on $\Gamma\cap[G,G]$, we therefore have $H^*(V;F_\chi)=0$.
		Hence, $H^*(\Gamma\setminus G;F_\chi)=0$ by the Leray--Serre spectral sequence, and $\tau_\RS(\Gamma\setminus G;F_\chi)=1$ according to \cite[Corollary~0.8]{LST98}.

		If $\chi$ vanishes on $\Gamma\cap[G,G]$, then there exists a flat complex line bundle $\tilde F$ over $B$ such that $F_\chi=p^*\tilde F$.
		As $\chi$ was assumed to be nontrivial, $\tilde F$ must be nontrivial as well.
		Hence, $H^*(B;\tilde F)=0$.
		Let $Z'$ denote a 1-parameter subgroup in the center $Z$ that intersects $\Gamma$ nontrivially.
		The canonical homomorphism $G\to G/[G,G]$ factors into a sequence of homomorphisms $G\to G/Z'\to G/Z\to G/[G,G]$.
		Modding out the lattice, this yields a tower of circle bundles 
		\begin{equation}\label{E:tower}
			\Gamma\setminus G\xrightarrow{p_3}E_2\xrightarrow{p_2}E_1\xrightarrow{p_1}B,
		\end{equation}
		such that $p=p_1\circ p_2\circ p_3$.
		Using the corresponding Gysin sequences we obtain, successively, $H^*(E_1;p_1^*\tilde F)=0$, $H^*(E_1;p_2^*p_1^*\tilde F)=0$, and $H^*(\Gamma\setminus G;F_\chi)=0$ for we have $p_3^*p_2^*p_1^*\tilde F=p^*\tilde F=F_\chi$.
		Moreover, $\tau_\RS(\Gamma\setminus G;F_\chi)=1$ according to \cite[Corollary~0.9]{LST98}.
	\end{proof}

	For a unitary representation $\rho\colon\Gamma\to U(k)$ we consider the positive real number
	\[
		R(\Gamma,\rho):=\frac{\|-\|_{\mathcal D_{\Gamma\setminus G},g_{\Gamma\setminus G},h_\rho}^{\sdet H^*(\Gamma\setminus G;F_\rho)}}{\|-\|_\RS^{\sdet H^*(\Gamma\setminus G;F_\rho)}}.
	\]
	If $\rho$ is acyclic, that is, if $H^*(\Gamma\setminus G;F_\rho)=0$, then $R(\Gamma,\rho)=\frac{\tau_\RS(\Gamma\setminus G;F_\rho)}{\tau(\Gamma\setminus G,\mathcal D_{\Gamma\setminus G},F_\rho,g_{\Gamma\setminus G},h_\rho)}$, by the very definition.
	Hence, combining the theorem with Lemma~\ref{L:RS.Tor} we obtain
	\begin{equation}\label{E:R.G.chi}
		R(\Gamma,\chi)=1
	\end{equation}
	for every nontrivial character $\chi\colon\Gamma\to U(1)$.
	According to \cite[Proposition~3.18]{H22} the quantity $R(\Gamma,\chi)$ depends continuously on $\chi$.
	Hence, \eqref{E:R.G.chi} remains true for all unitary characters $\chi$.

	If $\rho\colon\Gamma\to U(k)$ is irreducible then there exists a sublattice $\Gamma'$ in $\Gamma$ and a unitary character $\chi'\colon\Gamma'\to U(1)$ such that $\rho$ is isomorphic to the representation of $\Gamma$ induced by $\chi'$, see \cite[Lemma~1]{B73}.
	Hence, $R(\Gamma,\rho)=R(\Gamma',\chi')$ by \cite[Lemma~4.3(b)]{H22} and \eqref{E:R.G.chi} yields 
	\begin{equation}\label{E:R.G.rho}
		R(\Gamma,\rho)=1
	\end{equation}
	for irreducible $\rho$.
	Clearly, $R(\Gamma,\rho_1\oplus\rho_2)=R(\Gamma,\rho_1)\cdot R(\Gamma,\rho_2)$ for any two unitary representations $\rho_1$ and $\rho_2$ of $\Gamma$, see \cite[Lemma~4.3(a)]{H22}.
	Hence, \eqref{E:R.G.rho} remains true for all finite dimensional unitary representation $\rho$ of $\Gamma$.
	This completes the proof of the corollary.

\section*{Acknowledgments}
	This research was funded in whole or in part by the Austrian Science Fund (FWF) Grant DOI 10.55776/P31663.




\begin{thebibliography}{XX}

\bibitem{AQ22}
	P. Albin and H. Quan,
	\textit{Sub-Riemannian limit of the differential form heat kernels of contact manifolds.}
	Int. Math. Res. Not. IMRN 2022, no. \textbf{8}, 5818--5881.

\bibitem{AN14}
	D. An and P. Nurowski,
	\textit{Twistor space for rolling bodies.}
	Comm. Math. Phys. \textbf{326}(2014), 393--414.

\bibitem{AN18}
	D. An and P. Nurowski,
	\textit{Symmetric (2,3,5) distributions, an interesting ODE of 7th order and Plebański metric.}
	J. Geom. Phys. \textbf{126}(2018), 93--100.
	

\bibitem{BGS88}
	J.-M. Bismut, H. Gillet and C, Soul\'e,
	\textit{Analytic torsion and holomorphic determinant bundles. I. Bott-Chern forms and analytic torsion.}
	Comm. Math. Phys. \textbf{115}(1988), 49--78.

\bibitem{BZ92}
	J.-M.~Bismut and W.~Zhang,
	\textit{An extension of a theorem by Cheeger and M\"uller.}
	Ast\'erisque \textbf{205},
	Soci\'et\'e Math\'ematique de France, 1992.

\bibitem{BHN18}
	G. Bor, L. Hern\'andez Lamoneda, and P. Nurowski,
	\textit{The dancing metric, $G_2$-symmetry and projective rolling.}
	Trans. Amer. Math. Soc. \textbf{370}(2018), 4433--4481.

\bibitem{BM09}
	G. Bor and R. Montgomery,
	\textit{$G_2$ and the rolling distribution.}
	Enseign. Math. (2) \textbf{55}(2009), no. 1-2, 157--196.

\bibitem{B73}
	I. D. Brown, \textit{Representation of finitely generated nilpotent groups.}
	Pacific J. Math. \textbf{45}(1973), 13--26.

\bibitem{BENG11}
	R. L. Bryant, M. Eastwood, K. Neusser, and R. Gover,
	\textit{Some differential complexes within and beyond parabolic geometry.}
	Differential Geometry and Tanaka Theory (in honour of Professors Reiko Miyaoka and Keizo Yamaguchi), 13--40, Adv. Stud. Pure Math. \textbf{82}(2019).
	Preprint available at \url{https://arxiv.org/abs/1112.2142}

\bibitem{BH93}
	R. L. Bryant and L. Hsu,
	\textit{Rigidity of integral curves of rank 2 distributions.}
	Invent. Math. \textbf{114}(1993), 435--461.

\bibitem{CN09}
	A. \v Cap and K. Neusser,
	\textit{On automorphism groups of some types of generic distributions.}
	Differential Geom. Appl. \textbf{27}(2009), 769--779.

\bibitem{CSa09}
	A. \v Cap and K. Sagerschnig,
	\textit{On Nurowski’s conformal structure associated to a generic rank two distribution in dimension five.}
	J. Geom. Phys. \textbf{59}(2009), 901--912.

\bibitem{CS09}
	A. \v Cap and J. Slov\'ak,
	\textit{Parabolic geometries. I. Background and general theory.}
	Mathematical Surveys and Monographs \textbf{154}.
	American Mathematical Society, Providence, RI, 2009.

\bibitem{C10}
	E. Cartan,
	\textit{Les syst\`emes de Pfaff, \`a cinq variables et les \'equations aux d\'eriv\'ees partielles du second ordre.}
	Ann. Sci. \'Ecole Norm. Sup. \textbf{27}(1910), 109--192.

\bibitem{Ch77}
	J. Cheeger,
	\textit{Analytic torsion and Reidemeister torsion.}
	Proc. Natl. Acad. Sci. USA \textbf{74}(1977), 2651--2654.

\bibitem{Ch79}
	J. Cheeger,
	\textit{Analytic torsion and the heat equation.}
	Ann. of Math. \textbf{109}(1979), 259--322.

\bibitem{CG90}
	L. J. Corwin and F. P. Greenleaf, \textit{Representations of nilpotent Lie groups and their applications. Part I. Basic theory and examples.}
	Cambridge Studies in Advanced Mathematics \textbf{18}. 
	Cambridge University Press, Cambridge, 1990.

\bibitem{DH17}
	S. Dave and S. Haller, \textit{Graded hypoellipticity of BGG sequences.}
	Preprint \url{https://arxiv.org/abs/1705.01659v2}

\bibitem{DH19}
	S. Dave and S. Haller, \textit{On 5-manifolds admitting rank two distributions of Cartan type.}
	Trans. Amer. Math. Soc. \textbf{371}(2019), 4911--4929.

\bibitem{DH22}
	S. Dave and S. Haller, \textit{Graded hypoellipticity of BGG sequences.}
	Ann. Global Anal. Geom. \textbf{62}(2022) 721--789.

\bibitem{DH20}
	S. Dave and S. Haller, \textit{The heat asymptotics on filtered manifolds.}
	J. Geom. Anal. \textbf{30}(2020), 337--389.

\bibitem{D58}
	J. Dixmier, \textit{Sur les repr\'esentations unitaries des groupes de Lie nilpotents. III.} 
	Canad. J. Math. \textbf{10}(1958), 321--348.

\bibitem{E03}
	P. Epstein,
	\textit{Zur Theorie allgemeiner Zetafunctionen.} (German) 
	Math. Ann. \textbf{56}(1903), 615--644.


\bibitem{F22}
	V. Fischer,
	\textit{Asymptotics and zeta functions on compact nilmanifolds.}
	J. Math. Pures Appl. (9) \textbf{160}(2022), 1--28.

\bibitem{FT23}
	V. Fischer and F. Tripaldi,
	\textit{An alternative construction of the Rumin complex on homogeneous nilpotent Lie groups.}
	Adv. Math. \textbf{429}(2023), Paper No. 109192, 39 pp.

\bibitem{GPW17}
	A. R. Gover, R. Panai and T. Willse,
	\textit{Nearly K\"ahler geometry and (2,3,5)-distributions via projective holonomy.}
	Indiana Univ. Math. J. \textbf{66}(2017), 1351--1416.

\bibitem{H22}
	S. Haller, \textit{Analytic torsion of generic rank two distributions in dimension five.}
	J. Geom. Anal. \textbf{32}(2022), Paper No.~248, 66 pp.

\bibitem{H23}
	S. Haller, \textit{Regularized determinants of the Rumin complex in irreducible unitary representations of the (2,3,5) nilpotent Lie group.}
	Preprint: \url{https://arxiv.org/abs/2309.11159v2}

\bibitem{HS09}
	M. Hammerl and K. Sagerschnig,
	\textit{Conformal structures associated to generic rank 2 distributions on 5-manifolds --- Characterization and Killing-field decomposition.}
	SIGMA Symmetry Integrability Geom. Methods Appl. \textbf{5}(2009), Paper 081, 29 pp.

\bibitem{HS11}
	M. Hammerl and K. Sagerschnig,
	\textit{The twistor spinors of generic 2- and 3-distributions.}
	Ann. Global Anal. Geom. \textbf{39}(2011), 403--425.

\bibitem{H71}
	R. Howe,
	\textit{On Frobenius reciprocity for unipotent algebraic groups over $\Q$.}
	Amer. J. Math. \textbf{93}(1971), 163--172.


\bibitem{K62}
	A. A. Kirillov,
	\textit{Unitary representations of nilpotent Lie groups.}
	Russ. Math. Surv. \textbf{17}(1962), 53--104;
	English translation from Russian original Uspehi Mat. Nauk \textbf{17}(1962), 57--110.

\bibitem{K04}
	A. A. Kirillov, \textit{Lectures on the orbit method.}
	Graduate Studies in Mathematics \textbf{64}. 
	American Mathematical Society, Providence, RI, 2004.

\bibitem{K20}
	A. Kitaoka,
	\textit{Analytic torsions associated with the Rumin complex on contact spheres.}
	Internat. J. Math. \textbf{31}(2020), 2050112, 16 pp.

\bibitem{K22}
	A. Kitaoka,
	\textit{Ray-Singer torsion and the Rumin Laplacian on lens spaces.}
	SIGMA Symmetry Integrability Geom. Methods Appl. \textbf{18}(2022), Paper No. 091, 16 pp.

\bibitem{LNS17}
	T. Leistner, P. Nurowski and K. Sagerschnig,
	\textit{New relations between $G_2$ geometries in dimensions 5 and 7.}
	Internat. J. Math. \textbf{28}(2017), 1750094, 46 pp.

\bibitem{LST98}
	W. L\"uck, T. Schick and T. Thielmann, 
	\textit{Torsion and fibrations.}
	J. Reine Angew. Math. \textbf{498}(1998), 1--33.

\bibitem{M82}
	A. Melin,
	\textit{Lie filtrations and pseudo-differential operators.}
	Preprint, 1982.

\bibitem{M65}
	C. C. Moore,
	\textit{Decomposition of unitary representations defined by discrete subgroups of nilpotent groups.}
	Ann. of Math. \textbf{82}(1965), 146--182.

\bibitem{M78}
	W. M\"uller,
	\textit{Analytic torsion and R-torsion of Riemannian manifolds.}
	Adv. Math. \textbf{28}(1978), 233--305.

\bibitem{N05}
	P. Nurowski,
	\textit{Differential equations and conformal structures.}
	J. Geom. Phys. \textbf{55}(2005), 19--49.

\bibitem{R72}
	M. S. Raghunathan, \textit{Discrete subgroups of Lie groups.}
	Ergebnisse der Mathematik und ihrer Grenzgebiete \textbf{68}. 
	Springer-Verlag, New York-Heidelberg, 1972.

\bibitem{RS71}
	D. B. Ray and I. M. Singer,
	\textit{$R$-torsion and the Laplacian on Riemannian manifolds,}
	Adv. Math. \textbf{7}(1971), 145--210.

\bibitem{R71}
	L. F. Richardson, 
	\textit{Decomposition of the L2-space of a general compact nilmanifold.}
	Amer. J. Math. \textbf{93}(1971), 173--190.

\bibitem{R78}
	C. Rockland,
	\textit{Hypoellipticity on the Heisenberg group-representation-theoretic criteria.}
	Trans. Amer. Math. Soc. \textbf{240}(1978), 1--52.

\bibitem{R90}
	M. Rumin,
	\textit{Un complexe de formes diff\'erentielles sur les vari\'et\'es de contact.}
	C. R. Acad. Sci. Paris S\'er. I Math. \textbf{310}(1990), 401--404.

\bibitem{R94}
	M. Rumin,
	\textit{Formes diff\'erentielles sur les vari\'et\'es de contact.}
	J. Differential Geom. \textbf{39}(1994), 281--330.

\bibitem{R99}
	M. Rumin, 
	\textit{Differential geometry on C-C spaces and application to the Novikov-Shubin numbers of nilpotent Lie groups.}
	C. R. Acad. Sci. Paris S\'er. I Math. \textbf{329}(1999), 985--990.
	
\bibitem{R00}
	M. Rumin,
	\textit{Sub-Riemannian limit of the differential form spectrum of contact manifolds.}
	Geom. Funct. Anal. \textbf{10}(2000), 407--452.

\bibitem{R01}
	M. Rumin, \textit{Around heat decay on forms and relations of nilpotent Lie groups.}
	S\'eminaire de Th\'eorie Spectrale et G\'eom\'etrie 19, Ann\'ee 2000--2001, 123--164, S\'emin. Th\'eor. Spectr. G\'eom. \textbf{19}(2001), 
	Univ. Grenoble I, Saint-Martin-d'H\`eres, (2001)

\bibitem{R05}
	M. Rumin,
	\textit{An introduction to spectral and differential geometry in Carnot--Carath\'eodory spaces.}
	Rend. Circ. Mat. Palermo (2) Suppl. No. \textbf{75}(2005), 139--196.

\bibitem{RS12}
	M. Rumin and N. Seshadri, 
	\textit{Analytic torsions on contact manifolds.}
	Ann. Inst. Fourier (Grenoble) \textbf{62}(2012), 727--782.

\bibitem{S06}
	K. Sagerschnig,
	\textit{Split octonions and generic rank two distributions in dimension five.}
	Arch. Math. (Brno) \textbf{42}(2006), suppl, 329--339.

\bibitem{S08}
	K. Sagerschnig,
	\textit{Weyl structures for generic rank two distributions in dimension five.}
	PhD thesis, University of Vienna, Vienna, Austria, 2008.
	Available at \url{http://othes.univie.ac.at/2186/}.

\bibitem{SW17}
	K. Sagerschnig and T. Willse,
	\textit{The geometry of almost Einstein (2,3,5) distributions.}
	SIGMA Symmetry Integrability Geom. Methods Appl. \textbf{13}(2017), Paper No. 004, 56 pp.

\bibitem{SW17b}
	K. Sagerschnig and T. Willse,	
	\textit{The almost Einstein operator for (2,3,5) distributions.}
	Arch. Math. (Brno) \textbf{53}(2017), 347--370.

\bibitem{EY19}
	E. van~Erp and R. Yuncken,
	\textit{A groupoid approach to pseudodifferential calculi.}
	J. reine angew. Math. \textbf{756}(2019), 151--182.

\end{thebibliography}
\end{document}